\documentclass[a4paper,11pt,reqno]{article}
\usepackage[margin=1.2in]{geometry}
\usepackage{amsmath}
\usepackage{amsfonts}
\usepackage{amsthm}
\usepackage{amssymb}
\usepackage{tikz}
\usepackage[all]{xy}
\usepackage[colorlinks]{hyperref}
\usepackage{enumerate}
\usepackage{mathtools}
\usepackage{setspace}
\usepackage{tocloft}

\definecolor{darkgreen}{rgb}{0,0.5,0}
\definecolor{darkblue}{rgb}{0,0.1,0.5}

\hypersetup{
    colorlinks=true, 
    linkcolor=darkblue, 
    urlcolor=blue, 
    linktoc=all, 
    citecolor=darkgreen
    }

\newcommand{\doi[2]}{\href{http://dx.doi.org/#1}{#2}}
\theoremstyle{remark}
\newtheorem{theorem}{Theorem}[section]
\newtheorem*{theoremX}{Theorem}

\newtheorem{corollary}[theorem]{Corollary}

\newtheorem{definition}[theorem]{Definition}

\newtheorem{example}[theorem]{Example}

\newtheorem{proposition}[theorem]{Proposition}
\newtheorem*{propositionX}{Proposition}
\newtheorem{remark}[theorem]{Remark}

\newcommand{\Rep}{{\rm Rep}}

\newcommand{\cB}{\mathcal{B}}
\newcommand{\cC}{\mathcal{C}}
\newcommand{\cD}{\mathcal{D}}
\newcommand{\cF}{\mathcal{F}}
\newcommand{\cG}{\mathcal{G}}
\newcommand{\cH}{\mathcal{H}}
\newcommand{\cM}{\mathcal{M}}

\newcommand{\cV}{\mathcal{V}}
\newcommand{\cW}{\mathcal{W}}

\newcommand{\fg}{\mathfrak{g}}

\newcommand{\C}{\mathbb{C}}

\newcommand{\Z}{\mathbb{Z}}
\newcommand{\V}{\mathbb{V}}

\newcommand{\Vect}{\mathrm{Vect}}

\newcommand{\reg}{\mathrm{reg}}

\usepackage{lineno}
\newcommand{\marginparIf}[1]{  }

\begin{document}
\title{Characterizing braided tensor categories associated to logarithmic vertex operator algebras}
\author{Thomas Creutzig, Simon Lentner, and Matthew Rupert}

\maketitle

\begin{abstract}

Given a non-semisimple braided tensor category, with oplax tensor functors from known braided tensor categories, we ask : How does this knowledge characterize the tensor product and the braiding? We develop tools that address this question. In particular we prove that the associator is fixed by the oplax tensor functors, and we show that a distinguished role is played by the coalgebra structure on the image of theses tensor functors.

Our setup constrains the form of quasi bialgebras appearing in the logarithmic Kazhdan-Lusztig conjecture and it applies in particular to the representation categories of the triplet vertex algebras.
Here the two oplax tensor functors are determined by two free field realizations, and the coalgebras mentioned above are the Nichols algebras of type $\mathfrak{sl}_2$. 
We demonstrate in the case of $p=2$ that our setup completely determines the braided tensor category and the realizing quasi-triangular quasi-Hopf algebra is as anticipated in \cite{FGR2}.  
This proves the  logarithmic Kazhdan-Lusztig conjecture  for $p=2$, while for general $p$ it only remains to establish that our characterization provides a unique braided tensor category.
\end{abstract}

\quad\\

\setcounter{tocdepth}{2}
\tableofcontents

\section{Introduction}

There are many interesting connections between representation categories of vertex operator algebras and quantum groups associated to the same simple Lie algebra $\fg$. The most prominent example is the work of Kazhdan and Lusztig \cite{KL1,KL2,KL3,KL4}, where they proved a braided equivalence between representation categories of affine Lie algebras at generic level and quantum groups at an appropriate corresponding parameter $q$. 

A vertex algebra fulfilling sufficient finiteness conditions gives rise to a braided tensor category of representations \cite{Hu1, HLZ}.
Over the last two decades logarithmic theories received increased attention. The word logarithmic comes from conformal field theory and refers to logarithmic singularities in correlation functions, see \cite{CR} for an introduction. This happens if the representation category of the underlying vertex operator algebra is non-semisimple.  It is in general a difficult problem to completely understand the braided tensor category of a given vertex algebra. Our approach is to use structure of the vertex algebra and its tensor category that is fairly accessible in order to characterize a quasi Hopf algebra with braided equivalent tensor category. First examples of vertex algebras with non semi-simple representation categories were the triplet algebras $\cW(p)$ and our main focus are 
correspondences of triplet algebras and related vertex algebras and quasi Hopf modifications of quantum groups. 

Connections between the triplet vertex operator algebra and restricted quantum groups of $\mathfrak{sl}_2$ at roots of unity first appeared in the work of Feigin, Gainutdinov, Semikhatov, and Tipunin \cite{FGST1,FGST2,FGST3,FGST4}. They conjectured therein the ribbon equivalence of particular representation categories of $\mathcal{W}(p)$ and the small quantum group $u_q(\mathfrak{sl}_2)$
 at $2p$-th root of unity $q$, and an abelian equivalence between these categories was later claimed to be shown by Nagatomo and Tsuchiya \cite{NT}. Proofs were however lacking and details have finally and very recently been filled in by McRae and Yang \cite{MY}, which used the work of Adamovic, Milas and others \cite{TW, A, AM1, AM2, AM3, AM4, CJORY,CMY}. For a general Lie algebra $\mathfrak{g}$, the {\it Logarithmic Kazhdan Lusztig Conjecture} is a conjectured ribbon equivalence between the representation categories of higher rank analogs of the triplet algebras and a quantum group $u_q(\mathfrak{g})$ and a $2p$-th root of unity $q$ \cite{FT, AM5,L, S}. 

However, in this form the conjecture cannot be true, because at even roots of unity the category of representations of a quantum group is typically not a (non-semisimple) modular tensor category. For example for $\mathfrak{sl}_2,p>2$ this was demonstrated by Kondo and Saito \cite{KS}. It was then shown by Gainutdinov, Runkel and the first author that there exists a quasi Hopf algebra $\tilde{u}_q(\mathfrak{sl}_2)$  whose underlying algebra is  $u_q(\mathfrak{sl}_2)$ \cite{FGR2,CGR}, and whose representation category is a non-semisimple modular tensor category. This conjecture is a consequence of the conjectural correspondence between the singlet algebra $\mathcal{M}(p)$ and the unrolled restricted quantum group $u_q^H(\mathfrak{sl}_2)$ at $2p$-th root of unity $q$.
This conjecture first appeared in \cite{CGP}, and has been motivated in \cite{CM,CMR,CGR,Ru}. To be precise, it is expected that the category of finite-dimensional weight $u_q^H(\mathfrak{sl}_2)$-modules $\mathcal{C}^H$ should be ribbon equivalent to the smallest subcategory of singlet modules generated by irreducibles with respect to finite sums, tensor products, and quotients. In fact, it was this conjecture that motivated the work of \cite{CGR}. The triplet can be realized as a simple current extension of the singlet, and its representation category is therefore expected to correspond to the category of local modules $\mathrm{Rep}^0 \mathcal{A}_p$ for some commutative algebra object in $\mathcal{A}_p \in \mathcal{C}^H$. The quasi Hopf algebra constructed in \cite{CGR} is precisely the one which realizes $\mathrm{Rep}^0 \mathcal{A}_p$ as its own category of modules. The modified Kazhdan Lusztig conjecture was then that the representation category of this quasi Hopf modification is braided tensor equivalent to the category of the triplet algebra. \\

For arbitrary Lie algebras $\mathfrak{g}$ and even order of $q$ such a modified quasi Hopf algebra $\tilde{u}_q(\mathfrak{g})$ that gives rise to a modular tensor category was constructed by Gainutdinov, Ohrmann and the second author \cite{GLO} and from the perspective of de-equivariantization by Negron \cite{N}. Our construction proceeds in the following way, which is inspired by the Andruskiewitsch Schneider program \cite{AS} for classifying pointed Hopf algebras: It starts with the correct modular tensor category
of vector spaces graded by the Cartan part, an abelian group of here even order, which is a quasi Hopf algebra. On the vertex algebra side, this is the semisimple modular tensor category of representations of the lattice vertex algebra underlying the free field construction. Then one extends this category by Nichols algebras \cite{H} inside the category, which correspond the the Borel parts. On the vertex algebra side, these are the algebras of screening operators \cite{L}. As for Lie algebras, between these two categories there is an adjunctions of induction to Verma modules and restriction to weight spaces. On the vertex algebra side, these is reversed an adjunction of restriction from the lattice vertex algebra and induction.    \\

We aim to develop a general technology to prove braided tensor equivalences between categories of modules of vertex algebras and quasi Hopf algebras. Indeed, there are many more interesting vertex algebras that are expected to be related to quasi Hopf modifications of quantum groups. The most obvious examples are those which are closely related to the triplet such as the singlet algebra $\mathcal{M}(p)$, and $\mathcal B_p$ vertex algebras \cite{CRW, ACKR}. The singlet is the $U(1)$-orbifold of the triplet and the $\mathcal B_p$-algebras  are quantum Hamiltonian reductions of $\mathfrak{sl_{p-1}}$ at level $-(p-1)^2/p$ for the subregular nilpotent element \cite{ACGY}. Their Heisenberg coset is $\mathcal{M}(p)$ \cite{A2, CRW}. We mention here especially the $\mathcal B_p$-algebras as they have relaxed-highest weight modules and spectral flow twists thereof, i.e. modules that neither have finite dimensional conformal weight spaces nor is the conformal weight necessarily lower bounded. Affine vertex (super)algebras have similar representations (if the level is not a positive integer) and understanding their representation theory is of major importance.

\subsection{From vertex operator algebras to quasi Hopf algebras}

We aim to develop a general formalism to characterize the quasi Hopf algebra structure of representation categories of vertex operator algebras. 

\subsubsection{The vertex algebra set-up}

We start by listing assumptions that should hold for general classes of vertex algebras with non semi-simple representation category. 
\begin{enumerate}
	\item Let $\cW$ be a vertex operator algebra and $\cC= \Rep(\cW)$ a vertex tensor category of $\cW$-modules. In particulary it is a braided tensor category. 
	We also assume that every object in $\cC$ is rigid, has integral Frobenius-Perron dimension and that this category is locally finite. Let $U$ be an algebra and $\Psi_\cW :  \Rep(U) \rightarrow \cC$ an equivalence of abelian categories that preserves the Frobenius Perron dimensions.
	\item There is an abelian group $L$ and non-degenerate quadratic form $Q$ on $L$, such that the quasi-fiber functor $ \cC \rightarrow \Vect$ factors through $\Vect_L^Q$. 
	\item There is a family of  embeddings $\iota_a: \cW \rightarrow \cV$ of vertex operator algebras of finite index for $a= 1, \dots, n$, such that $\Rep(\cV) \cong \Vect_L^Q$.
	In particular there are induction functors $\cF_a: \Rep(\cW) \rightarrow \Rep(\cV)^{\text{tw}}$ to a category $\Rep(\cV)^{\text{tw}}$ that contains $\Rep(\cV)$ as subcategory. The right adjoint $\cG_a$ is just the restriction functor that forgets the structure of the larger algebra $\cV$. The right adjoint of a monoidal functor is oplax and so in particular it provides an lax tensor functor $\cG_a : \Rep(\cV)^{\text{tw}} \rightarrow \Rep(\cW)$. Especially the restriction to $\Rep(\cV)$ of $\cG$ is oplax. We require that
	\[
	\bigotimes_{a=1}^n \cG_a( \Rep(\cV)) 
	\]
	is a projective generator of $\Rep(\cW)$. 
	\item  The semi-simple part of the ribbon twist on objects in $\Rep(\cW)$ is known.
\end{enumerate}
We now comment on these assumptions and explain that they are satisfied in the main example of current interest, the triplet $\cW(p)$. The theory of vertex tensor categories is due to Huang, Lepowsky and Zhang \cite{HLZ} and
there are a few general Theorems ensuring the existence of vertex tensor categories. 
\begin{enumerate}
	\item $C_2$-cofinite vertex operator algebras of positive energy \cite{Hu1}.
	\item Vertex operator algebras with the properties that all irreducible ordinary modules are $C_1$-cofininte and all generalized Verma modules are of finite length \cite{CY}.
	\item Let $V \subset W$ be vertex operator algebras, such that their conformal vectors coincide. If $W$ is an object in a suitable completion of a vertex tensor category $\cC_V$ of $V$-modules, then the category of $W$-modules that lie in $\cC_V$ is a vertex tensor category as well \cite{CKM, CMY2}.
\end{enumerate}	
In particular the triplet algebra $\cW(p)$ is $C_2$-cofinite \cite{AM2} and obviously of positive energy. It is believed but unproven that the category of ordinary modules of the singlet algebra $\cM(p)$  falls into the second type, see section 6 of \cite{CMR}. Currently only vertex tensor category is known for a subcategory \cite{CMY} by using the last approach and  vertex tensor category results of ordinary modules of the Virasoro algebra \cite{CJORY}.
There is no general rigidity Theorem for non-rational VOAs, however it has been explicitly verified for $\cW(p)$ \cite{TW} and for  the just mentioned subcategory of $\cM(p)$ \cite{CMY}.
Integrality of Frobenius-Perron dimensions follows from fusion rules that are determined in both cases \cite{TW, CMY, MY}.
The abelian equivalence between $\Rep(\cW(p))$ and  $\Rep(u_q(\mathfrak{sl}_2))$ is settled as mentioned above. The Cartan subalgebra $C \subset u_q(\mathfrak{sl}_2)$ has as representation category $\Vect_L^Q$.
Verifying that the quasi-fiber functor factors through $\Vect_L^Q$ is just a fusion rule computation that in our case follows from the knowledge that fusion rules in $\Rep(\cW(p))$ and  $\Rep(u_q(\mathfrak{sl}_2))$ coincide for the monoidal structure on $\Rep(u_q(\mathfrak{sl}_2))$  determinend in \cite{CGR} and since the fiber functor for $\Rep(u_q(\mathfrak{sl}_2))$ clearly factors through $\Vect_{\mathbb Z_{2p}}^Q$ the same must be true for $\Rep(\cW(p))$.
Finally, the ribbon twist on VOA-modules is determined by conformal weights, which are among the first quantities one is usually able to determine. 
It remains the family of oplax tensor functors. We prove 
\begin{theoremX}[\ref{thm:oplax}]
	Let $\cC_{\cW(p)}$ be the category of modules of the triplet $\cW(p)$ vertex operator algebra  for integer $p>1$ and let $\cC_{V_L}$ be the category of modules of the lattice vertex algebra of the lattice $L= \sqrt{2p}\mathbb Z$.  
	Then there exist (not additive) subcategories $\cV, \overline{\cV}$ of $\cC_{\cW(p)}$ with the properties that 
	$\cV \boxtimes \overline{\cV}$ is a projective generator of $\cC_{\cW(p)}$ and 
	there are surjective oplax tensor functors $\cG : \cC_{V_L} \rightarrow \cV$ and $\cG^\sigma : \cC_{V_L} \rightarrow \overline{\cV}$.
\end{theoremX}
We prove this by twisting the usual embedding of the triplet in $V_L$ by an automorphism and then compute relevant fusion rules. We used the results of \cite{CMY}.  We find, more explicitly spoken, the fusion product of Verma module and twisted (opposite) Verma module is projective.

\subsubsection{The characterization Theorem}

With the above assumptions on the vertex algebra $\cW$ and its associated algebra $U$ in place, we consider by Tannakian reconstruction a quasi bialgebra structure on $U$, such that $\cC = \Rep(\cW)$ is tensor equivalent $\Rep(U)$. First we prove a slightly enhanced ``relative" version of Tannakian reconstruction relevant to our situation:\\


\begin{propositionX}[\ref{prop2}]
	Let $\cC, \cD$ be a finite tensor categories such that there exists a quasi-fiber functor $F_{\cD}:\cD \to \mathrm{Vec}$ and an essentially surjective quasi-tensor functor $G:\cC \to \cD$. Then we have an inclusion $C \subset U$ of the associated quasi bialgebra $U=\mathrm{End}(F_{\cD} \circ G)$, $C=\mathrm{End}(F_{\cD})$ preserving the coproduct, counit, and algebra structure. 
\end{propositionX}

Then we prove that if the category can be ``covered" with images of oplax tensor functors, then the previous inclusion is an inclusion of quasi bialgebras resp. $G$ is a true tensor functor. In essence, the proof reflects the fact that the regular representation of a bialgebras is a coalgebra in the respective tensor category, but for a true quasi bialgebra this fails. \\

\begin{theoremX}[\ref{Catthm}]
	Let $\cC, \cD$ be a finite tensor categories such that there exists a quasi-fiber functor $F_{\cD}:\cD \to \mathrm{Vec}$ and an essentially surjective quasi-tensor functor $G:\cC \to \cD$ as in Proposition \ref{prop2}. 	Suppose we have oplax monoidal functors $\mathbb{V}_k:\mathrm{Rep}(C) \to \mathrm{Rep}(U)$, $k=1,...,n$ such that $\bigotimes_k \mathbb{V}_k(C_{reg})$ is a projective generator of $\mathrm{Rep}(U)$, such that in addition the coproduct $\delta$ induced by the coproduct of the regular $C$-$C$-bimodule $C_{reg}$ sends the image of a generator $1$ to a  $U$-$U$-bimodule generator.
	Then there exist a twist $U^J$ of $U$, such that $\Phi_{U^J}=\Phi_C$ and thus $C=\mathrm{End}(F_{\cD})$ is a sub quasi bialgebra of $U^J$. Moreover there is a unique choice of such a twist such that $\delta(1)=1\otimes 1$.
\end{theoremX}

 The technical use of this theorem is that it fixes a twist representative of $U$ and that it provides us a-priori with $U$-$C$-comodule coalgebras $\mathbb{V}_k(C_{reg})$ with a known associator $\Phi_U=\Phi_C$. Eventually, we can then determine these coalgebras explicitly and thus our category in question. 

\subsubsection{The correspondence of $\cW(2)$ and $\tilde{u}_i(\mathfrak{sl}_2)$}

Now we consider the vertex algebra $\cW=\cW(2)$, the algebra $U=u_i(\mathfrak{sl}_2)=\tilde{u}_i(\mathfrak{sl}_2)$, its Cartan subalgebra $C=\mathbb{C}[\mathbb{Z}_4]$, and the lattice $L=\sqrt 2 A_1$. Denote by $\V,\overline{\V}:\mathrm{Rep}(C) \to \Rep(U)$ the oplax tensor functors which send an irreducible $C$-module $\mathbb{C}_k$ of weight $i^k$ to the corresponding Verma and opposite Verma modules $M_k$ and $\overline{M}_k$ of highest weight $i^k$. Then we have the following

\begin{theorem}\label{classif}
	The abelian category $\Rep(U)$ admits up to equivalence a unique structure of a braided tensor category such that $\V,\overline{\V}:\Rep(C) \to \Rep(U)$ are braided oplax tensor functors with the braiding on $C$ given by the quadratic form $Q(k)=\beta^{(k^2)}$, for fixed $\beta^4=-1$.
\end{theorem}

\noindent It follows from Theorem \ref{thm:oplax} that the triplet algebra $\mathcal{W}(2)$ admits two such oplax tensor functors $\V,\overline{\V}:\Rep V_L \to \Rep \mathcal{W}(2)$ where the $Q$ is determined by the quadratic form of the lattice $L=\sqrt{2} A_1$, so $\beta=e^{2\pi i(1/8)}$. It then follows that the abelian category $\mathrm{Rep}(U)$ of the Kazhdan-Lusztig dual $U$ of $\mathcal{W}(2)$ must admit the functors described in Theorem \ref{classif}. We therefore obtain as a corollary a proof of the Kazhdan-Lusztig conjecture for $\mathcal{W}(2)$:
\begin{corollary}
There is a braided tensor equivalence of  $\Rep(\mathcal{W}(2))$  to the representation category of the previously constructed quasi-Hopf algebra for this value of $\beta$. The quasi Hopf algebra coincides with the one constructed in \cite{FGR2}.
\end{corollary}

The proof has the following steps:
\begin{itemize}
\item We have already established that $C$ is a sub quasi bialgebra of $U$, identified with the Cartan part $\C[K]/(K^ 4-1)$. We have four possibilities parametrized by $\beta^4=-1$, two of which are equivalent.
\item In Proposition \ref{prop_Vreg} we determine the possible coproducts on the $U$-$C$-coalgebras $\V(C_{reg}),\delta$ and $\overline{\V}(C_{reg}),\bar{\delta}$ fulfilling in addition $\delta(1)=\bar{\delta}(1)=1\otimes 1$. What we find matches the expectation of a dual Nichols algebra of type $\mathfrak{sl}_2$ (a truncated polynomial ring) and a family of twists of the Cartan part parametrized by $c^a,\bar{c}^a$.
\item From this we can directly compute in Proposition \ref{prop_EF} all coproducts  $\Delta(E),\Delta(F)$. Some general arguments show that there are no mixed $E$-$F$-terms, so the coproducts are determined by the previous steps and their compatibility implies a relation between $c^a,\bar{c}^a$. We also use the existence of an antipode for a nonzero condition.
\item In Proposition \ref{prop_automorphism} we apply algebra automorphisms to reduce $\Delta(F)$ to a standard form and find a $1$-parameter family of coproducts $\Delta(E)$ parametrized by a scalar $d$ (and a sign $\epsilon$ that disappears after switching $E,F$). These are precisely all quasi Hopf algebras resp. tensor categories with the desired properties. For $d=\pm i$ we recover the quasi Hopf algebra in \cite{FGR2} (for the negative sign after switching $K,K^ {-1}$)
\item In Proposition \ref{braiding} we compute all extensions of the $R$-matrix of $C$ to $U$. We find that our previous quasi Hopf algebras are only braidable if $d=\pm i$, which concludes our result.     
\end{itemize}


\subsection{Outlook}

We have seen that our characterization of the category of the $\cW(2)$-algebra already uniquely up to twist equivalence determines the quasi Hopf algebra structure on $\Rep(\tilde{u}_i(\mathfrak{sl}_2))$. It is work in progress to show that the same happens for general $\cW(p)$. We hope to report on this in the near future. 

Next, one would like to settle the equivalence between the singlet algebras and the unrolled quantum groups. For this the vertex tensor category results of \cite{CMY} need to be extended to the complete category of ordinary modules. We remark that none of our results require finiteness assumptions on the category, as long as we start with a known equivalence of abelian categories to representations over an algebra, finite or infinite dimensional.

More generally, we would like to apply our techniques to the logarithmic Kazhdan Lusztig conjecture and beyond: Many vertex algebras characterized as intersections of kernels of sets of screening charges acting on lattice vertex algebras, Heisenberg vertex algebras, or some more complicated vertex algebras.
The  logarithmic Kazhdan Lusztig conjecture is concerned with the Feigin-Tipunin algebras \cite{FT, S}. These are the natural higher rank generalizations of the triplet $\cW(p)$, and their representation category is conjectured to be equivalent to the representation category of a corresponding quasi quantum group. Higher rank generalizations of the singlet and $\mathcal B_p$-algebras also exist \cite{CM2, C}, and they are equally nicely characterized as kernels of screenings on free field vertex algebras. Before applying our technology to these cases much of the representation theory of these algebras has to first be studied. 
There is however already one very interesting example whose representation theory is well enough understood, namely the affine vertex superalgebra of  $\mathfrak{gl}_{1|1}$ \cite{CMY}. Unlike the Lie algebra case, its category of ordinary modules has uncountable infinitely many inequivalent simple objects. We note that nonetheless we will try to employ our technology to determine the quasi Hopf superalgebra whose representation category is braided equivalent to the category of ordinary modules of the affine vertex superalgebra of  $\mathfrak{gl}_{1|1}$. The underlying algebra must of course be the quantum supergroup of $\mathfrak{gl}_{1|1}$, i.e. we expect the linear equivalence to be easy to see. This example would thus provide a first Kazhdan-Lusztig correspondence for an affine vertex superalgebra and a quantum supergroup. Again it is clear from \cite{L} that the respective screening algebras are the Nichols algebras and it is natural to expect that they give precisely the coalgebra structures on the image of the oplax tensor functors. Other more exotic Nichols algebras over abelian groups or in more complicated braided tensor categories point to more exotic examples beyond super Lie algebras.  \\

\begin{remark}
	Terry Gannon and Cris Negron informed us that they proved the equivalence between $\cW(p)$-mod and the quasi Hopf modification of $\tilde{u}_q(\mathfrak{sl}_2)$-mod at $2p$-th root of unity \cite{GN}.  Their approach uses a universal property of the Temperly-Lieb category \cite{O} and de-equivariantization \cite{N} and is very different from ours. 
\end{remark}

\section{The Singlet and Triplet Vertex Operator Algebras}

Let $\cC_{\cW(p)}$ be the category of modules of the triplet $\cW(p)$ vertex operator algebras  for integer $p>1$ and let $\cC_{V_L}$ be the category of modules of the lattice vertex algebra of the lattice $L= \sqrt{2p}\mathbb Z$.  
The aim of this section is to prove the following Theorem, which summarizes Corollary \ref{cor:projective} and Proposition \ref{prop:oplax}.
\begin{theorem}\label{thm:oplax}
	There exist subcategories $\cV, \overline{\cV}$ of $\cC_{\cW(p)}$ with the properties that 
	$\cV \boxtimes \overline{\cV}$ is a projective generator of $\cC_{\cW(p)}$ and 
	there are surjective oplax tensor functors $\cG : \cC_{V_L} \rightarrow \cV$ and $\cG^\sigma : \cC_{V_L} \rightarrow \overline{\cV}$.
\end{theorem}

We review the singlet $\cM(p)$ and triplet $\cW(p)$ vertex operator algebras  for integer $p>1$. We restrict to the category $\mathcal C^0_{\cM(p)}$ of \cite{CMY}, since the category $\cC_{\cW(p)}$ of the triplet algebra is obtained from this one via vertex algebra extension. Moreover the category $\mathcal C^0_{\cM(p)}$ is braided monoidal and rigid \cite{CMY} and most fusion rules are known. We will determine missing ones in a moment. 
Let $\alpha_+ = \sqrt{2p}, \alpha_- = -\sqrt{2/p}$, $\alpha_0 = \alpha_+ + \alpha_-$ and 
\[
\alpha_{r, s} = \frac{1-r}{2} \alpha_+ + \frac{1-s}{2}\alpha_-
\]
Then the simple modules of the singlet algebra are denoted by $M_{r, s}$ for integers $r, s$ and in addition $1\leq s \leq p$. 
The top level of these modules has conformal weight
\[
h_{r, s} = \frac{\alpha_{r, s}(\alpha_{r, s}- \alpha_0)}{2}
\]
and so especially the ribbon twist on these modules is $e^{2\pi i h_{r, s}}\text{Id}_{M_{r, s}}$. 
The projective cover of $M_{r, s}$ is denoted by $P_{r, s}$.
The modules $M_{r, p}$ are projective, $M_{r, p} = P_{r, p}$. There are then further modules $F_{r, s}$ and $\overline{F}_{r, s}$ uniquely characterized by  fitting in the non-split exact sequence
\begin{equation}\label{ses1}
0 \rightarrow M_{r, s} \rightarrow F_{r, s} \rightarrow M_{r+1, p-s} \rightarrow 0, \qquad 
0 \rightarrow M_{r, s} \rightarrow \overline{F}_{r, s} \rightarrow M_{r-1, p-s} \rightarrow 0
\end{equation}
for $s \neq p$ and $F_{r, p} \cong \overline{F}_{r, p} \cong M_{r, p} \cong P_{r, p}$.
The $F_{r, s}$ and $\overline{F}_{r, s}$ are submodules and quotients of $P_{r, s}$ and fit in the non-split exact sequences
\begin{equation}\label{ses2}
0 \rightarrow F_{r, s} \rightarrow P_{r, s} \rightarrow F_{r-1, p-s} \rightarrow 0, \qquad 
0 \rightarrow \overline{F}_{r, s} \rightarrow P_{r, s} \rightarrow \overline{F}_{r+1, p-s} \rightarrow 0
\end{equation}
for $s\neq p$. The $M_{n+1, 1}$ are simple currents and satisfy the fusion rules
\begin{equation}
\begin{split}
M_{n+1, 1} \boxtimes M_{r, s} &\cong M_{r+n, s} \\
M_{n+1, 1} \boxtimes P_{r, s} &\cong P_{r+n, s} \\
M_{n+1, 1} \boxtimes F_{r, s} &\cong F_{r+n, s} \\
M_{n+1, 1} \boxtimes \overline{F}_{r, s} &\cong \overline{F}_{r+n, s}.
\end{split}
\end{equation}
These fusion rules are proven in \cite{CMY}, see especially section 3.2. The only one not stated there is the last one, but it follows immediately from simple currents preserving non-split short exact sequences, see Proposition 2.5 of \cite{CKLR}.

The triplet algebra $\cW(p)$ has simple modules $W_{r, s}$ with $r \in \{1, 2\}$ and $s \in \{ 1, \dots, p\}$. 
The ribbon twist on these modules is $e^{2\pi i h_{r, s}}\text{Id}_{W_{r, s}}$.
The module $W_{r, p}$ is projective and otherwise one has the extensions characterized by the non-split exact sequences 
\begin{equation}\label{Wses1}
0 \rightarrow W_{r, s} \rightarrow V_{r, s} \rightarrow W_{3-r, p-s} \rightarrow 0, \qquad 
0 \rightarrow W_{r, s} \rightarrow \overline{V}_{r, s} \rightarrow W_{3-r, p-s} \rightarrow 0.
\end{equation}
The modules $V_{r, s}$ and $\overline{V}_{r, s}$ are constructed as follows. 
The triplet algebra is 
\[
\cW(p) \cong \bigoplus_{n \in \mathbb Z} M_{1+2n, 1}
\]
as a singlet module. In particular there is an induction functor $\cF : \cC^0_{\cM(p)} \rightarrow \cC_{\cW(p)}$ that satisfies \cite{CMY}
\begin{equation}
\begin{split}
\cF(M_{r, s}) &= W_{\bar r, s} \\
\cF(P_{r, s}) &= R_{\bar r, s} \\
\cF(F_{r, s}) &= V_{\bar r, s} \\
\cF(\overline{F}_{r, s}) &= \overline{V}_{\bar r, s} \\
\end{split}
\end{equation}
Here $\bar r = 2$ if $r$ is even and $\bar r =1$ if $r$ is odd. The modules $R_{\bar r, s}$ are all projective and in particular the induction of a projective module is always projective. 
The last two can be taken as the definiton of the corresponding triplet modules and the short-exact sequences follow from exactness of $\cF$, see \cite{CKM}. The right adjoint of $\cF$ is the forgetful functor $\cG$. It satisfies
\[
\cG(\cF(M)) \cong \bigoplus_{n \in \mathbb Z} M_{2n+1, 1} \boxtimes M
\]
for an object $M$ in  $\cC^0_{\cM(p)}$. The important properties are that $\cF$ is monoidal and that Frobenius reciprocity holds, that is
\begin{equation}
\begin{split}
\cF(M \boxtimes M') &\cong \cF(M) \boxtimes \cF(M')\quad \text{and} \quad \\
\text{Hom}_{\cC_{\cW(p)}}(\cF(M), X) &\cong \text{Hom}_{\text{Ind}(\cC^0_{\cM(p)})}(M, \cG(X))
\end{split}
\end{equation}
for objects $M, M'$ in $\cC^0_{\cM(p)}$ and $X$ in $\cW(p)$. Here $\text{Ind}(\cC^0_{\cM(p)})$ is the direct limit completion of $ \cC^0_{\cM(p)}$ \cite{CMY2} (alternatively one can also take the direct sum completion of \cite{AR}). 
Frobenius reciprocity tells us that $V_{r, s}$  and  $\overline{V}_{r, s}$ are not isomorphic since
\[
\text{Hom}_{\cC_{\cW(p)}}(V_{r, s} , \overline{V}_{r, s}) \cong  \text{Hom}_{\cC_{\cW(p)}}(\cF(F_{r, s}) , \cF(\overline{V}_{r, s})) 
=   \text{Hom}_{\text{Ind}(\cC^0_{\cM(p)})}(F_{r, s}, \bigoplus_{n \in \mathbb Z}  \overline{F}_{r+2n, s} ) = 0.
\]

\subsection{Fusion Rules}

We need the fusion rules for $F_{r, s}$ with $\overline{F}_{r', s'}$. 
Define 
\[
P_{r', s', r, s} := \bigoplus_{\substack{ \ell =2p+1-s-s' \\  \ell +s + s' \ \text{odd}}}^{p} P_{r+r'-1, \ell}.
\]
\begin{theorem}
	The fusion rules 
	\begin{equation}
	\begin{split}
	M_{r', s'} \boxtimes M_{r, s} &=P_{r', s', r, s} \oplus  \bigoplus_{\substack{ \ell = |s-s'| +1 \\  \ell +s + s' \ \text{odd}}}^{\text{min}\{s+s'-1, 2p-1-s-s' \}} M_{r+r'-1, \ell} \\
	M_{r', s'} \boxtimes F_{r, s} &=P_{r', s', r, s} \oplus P_{r', s', r+1, p-s } \oplus  \bigoplus_{\substack{ \ell = |s-s'| +1 \\  \ell +s + s' \ \text{odd}}}^{\text{min}\{s+s'-1, 2p-1-s-s' \}} F_{r+r'-1, \ell} \\
	M_{r', s'} \boxtimes \overline{F}_{r, s} &=P_{r', s', r, s} \oplus P_{r', s', r-1, p-s } \oplus  \bigoplus_{\substack{ \ell = |s-s'| +1 \\  \ell +s + s' \ \text{odd}}}^{\text{min}\{s+s'-1, 2p-1-s-s' \}} \overline{F}_{r+r'-1, \ell} \\
	\end{split}
	\end{equation}
	hold.
\end{theorem}
\begin{proof}
	The first fusion rule of the Theorem is part of the main Theorem of \cite{CMY}. The other rules follow from these together with rigidity and the structure of the projective modules. 
	
	We start with the fusion rules $M_{n+1,1} \boxtimes M_{r, s} = M_{r+n, 1}$, i.e. the $M_{n_1, 1}$ are simple currents. These preserve non-split exact sequences \cite[Prop. 2.5 ]{CKLR} and thus $M_{n+1,1} \boxtimes F_{r, s} = F_{r+n, 1}$.
	Next, we use
	\begin{equation}
	\begin{split}
	M_{1, 2} \boxtimes M_{r, s} &= \begin{cases}  M_{r, 2} & \quad s=1 \\ M_{r, s-1} \oplus M_{r, s+1} &\quad 1<s<p \\ P_{r, p-1} &\quad s=p \end{cases} \\
	M_{1, 2} \boxtimes P_{r, s} &= \begin{cases}  P_{r, 2} \oplus P_{r+1, p} \oplus P_{r-1, p} & \quad s=1 \\ P_{r, s-1} \oplus P_{r, s+1} &\quad 1<s<p-1 \\ P_{r, p-2} \oplus 2P_{r, p} &\quad s=p-1 \end{cases}
	\end{split}
	\end{equation}
	Consider the case $s=1$, We tensor \eqref{ses1} and \eqref{ses2} with $M_{1, 2}$ and since $\mathcal C^0_{\cM(p)}$ is rigid we obtain the short exact sequences
	\begin{equation}\label{ses3}
	\begin{split}
	&0 \rightarrow M_{r, 2} \rightarrow M_{1, 2} \boxtimes F_{r, 1} \rightarrow M_{r+1, p-2} \oplus M_{r+1, p} \rightarrow 0, \qquad  \\
	&0 \rightarrow M_{r-1, p-2} \oplus M_{r-1, p}  \rightarrow M_{1, 2} \boxtimes F_{r-1, p-1} \rightarrow M_{r, 2} \rightarrow 0, \qquad  \\
	&0 \rightarrow M_{1, 2} \boxtimes F_{r, 1} \rightarrow P_{r, 2} \oplus P_{r+1, p} \oplus P_{r-1, p} \rightarrow M_{1, 2} \boxtimes F_{r-1 , p-1}  \rightarrow 0
	\end{split}
	\end{equation}
	Recall that $M_{r \pm 1, p} \cong P_{r \pm1, p}$ and hence we get the fusion rule
	\[
	M_{1, 2} \boxtimes F_{r, 1} = F_{r, 2} \oplus  P_{r+1,p}, \qquad M_{1, 2} \boxtimes F_{r, p-1} = F_{r, p-2} \oplus  P_{r,p}
	\]
	Let now $1 < s < p-1$. Tensoring \eqref{ses1} and \eqref{ses2} with $M_{1, 2}$ we obtain the non-split  short exact sequences
	\begin{equation}\label{ses3}
	\begin{split}
	&0 \rightarrow M_{r, s-1} \oplus M_{r, s+1} \rightarrow M_{1, 2} \boxtimes F_{r, s} \rightarrow M_{r+1, p-s-1} \oplus M_{r+1, p-s+1} \rightarrow 0, \qquad  \\
	&0 \rightarrow M_{r-1, p-s-1} \oplus M_{r-1, p-s+1}  \rightarrow M_{1, 2} \boxtimes F_{r-1, p-s} \rightarrow M_{r, s-1} \oplus M_{r, s+1} \rightarrow 0, \qquad  \\
	&0 \rightarrow M_{1, 2} \boxtimes F_{r, s} \rightarrow P_{r, s-1} \oplus P_{r, s+1}  \rightarrow M_{1, 2} \boxtimes F_{r-1 , p-s}  \rightarrow 0
	\end{split}
	\end{equation}
	From which we get 
	\[
	M_{1, 2} \boxtimes F_{r, s} = F_{r, s-1} \oplus  F_{r, s+1}.
	\]
	In a rigid tensor category projective modules form a tensor ideal. Denote by $\cH$ the functor to the quotient category. It is monoidal.
	We have 
	\begin{equation}\label{ses3}
	\begin{split}
	\cH(M_{1, 2}) \boxtimes \cH(M_{r, s}) &=  \begin{cases}  \cH(M_{r, 2}) & \quad s=1 \\ \cH(M_{r, s-1}) \oplus \cH(M_{r, s+1}) &\quad 1<s<p \end{cases}, \\
	\cH(M_{1, 2}) \boxtimes \cH(F_{r, s}) &=  \begin{cases}  \cH(F_{r, 2}) & \quad s=1 \\ \cH(F_{r, s-1}) \oplus \cH(F_{r, s+1}) &\quad 1<s<p \end{cases}.
	\end{split}
	\end{equation}
	Since the fusion rules $M_{r', s'} \boxtimes M_{r, s}$ are uniquely determined by associativity, commutativity, the simple currents $M_{n+1, 1}$, and the fusion rules of $M_{1, 2} \boxtimes M_{r, s}$, the same must be true for the fusion rules in the quotient by projective modules and hence the same must be true for the fusion rules $M_{r,' s'} \boxtimes F_{r, s}$, i.e. we get 
	\[
	M_{r', s'} \boxtimes F_{r, s} =P^+_{r', s', r, s}  \oplus  \bigoplus_{\substack{ \ell = |s-s'| +1 \\  \ell +s + s' \ \text{odd}}}^{\text{min}\{s+s'-1, 2p-1-s-s' \}} F_{r+r'-1, \ell}
	\]
	for some projective module $P^+_{r', s', r, s}$. Comparing with $M_{r', s'} \boxtimes ( M_{r, s} \oplus M_{r+1, p-s})$ then gives the claim. The argument for the fusion rule $M_{r', s'} \boxtimes \overline{F}_{r, s}$ is essentially the same and we don't repeat it.  
\end{proof}
A few special cases are 
\begin{equation}
\begin{split}
M_{r',1} \boxtimes F_{r, 1} &= F_{r+r'-1, 1} \\
M_{r'-1,p-1} \boxtimes F_{r, 1} &= F_{r+r'-2, p-1} \oplus \bigoplus_{\substack{\ell = 3 \\ \ell \ \text{odd}}}^p  P_{r+r'-1, \ell}\\
M_{r,1} \boxtimes \overline{F}_{r', 1} &= \overline{F}_{r+r'-1, 1} \\
M_{r+1,p-1} \boxtimes \overline{F}_{r, 1} &= \overline{F}_{r+r', p-1} \oplus \bigoplus_{\substack{\ell = 3 \\ \ell \ \text{odd}}}^p  P_{r+r'-1, \ell}
\end{split}
\end{equation}
Using the short exact sequences characterizing $F_{r, 1}$ and $\overline{F}_{r', 1}$ we get the two short exact sequences for $F_{r, 1} \boxtimes \overline{F}_{r', 1}$
\begin{equation}
\begin{split}
E_1 &=   0 \rightarrow M_{r', 1} \boxtimes F_{r, 1} \rightarrow \overline{F}_{r', 1}  \boxtimes  F_{r, 1} \rightarrow M_{r'-1, p-1} \boxtimes F_{r, 1} \rightarrow 0 \\
&= 0 \rightarrow  F_{r+r'-1, 1} \rightarrow \overline{F}_{r', 1}  \boxtimes  F_{r, 1} \rightarrow F_{r+r'-2, p-1} \oplus \bigoplus_{\substack{\ell = 3 \\ \ell \ \text{odd}}}^p  P_{r+r'-1, \ell} \rightarrow 0 \\
E_2 &=   0 \rightarrow \overline{F}_{r', 1} \boxtimes M_{r, 1} \rightarrow \overline{F}_{r', 1}  \boxtimes  F_{r, 1} \rightarrow \overline{F}_{r', 1} \boxtimes M_{r+1, p-1} \rightarrow 0 \\
&= 0 \rightarrow  \overline{F}_{r+r'-1, 1} \rightarrow \overline{F}_{r', 1}  \boxtimes  F_{r, 1} \rightarrow\overline{F}_{r+r', p-1} \oplus \bigoplus_{\substack{\ell = 3 \\ \ell \ \text{odd}}}^p  P_{r+r'-1, \ell} \rightarrow 0 
\end{split}
\end{equation}
It follows that $\overline{F}_{r', 1}  \boxtimes  F_{r, 1}$ equals $\bigoplus\limits_{\substack{\ell = 3 \\ \ell \ \text{odd}}}^p  P_{r+r'-1, \ell}$ plus a summand that has the same composition factors as $F_{r+r'-1, 1} \oplus F_{r+r'-2, p-1}$, that is $M_{r+r'-1, 1}$ with multiplicity two and $M_{r+r'-2, p-1}$ and $M_{r+r', p-1}$ once as composition factors. Moreover this module has to contain both $F_{r+r'-1, 1}$ and  $\overline{F}_{r+r'-1, 1}$ as submodules and both $F_{r+r'-2, p-1}$ and $\overline{F}_{r+r', p-1}$ as quotients. Clearly the only possibility for this is $P_{r+r'-1, 1}$, i.e.
\begin{equation}\label{fus1}
\overline{F}_{r', 1}  \boxtimes  F_{r, 1}  = \bigoplus_{\substack{\ell = 1 \\ \ell \ \text{odd}}}^p  P_{r+r'-1, \ell} .
\end{equation}
Since  $F_{r, s}$ is a summand of $M_{1, s} \boxtimes F_{r, 1}$ and $\overline{F}_{r', s'}$ is a summand of $M_{1, s'} \boxtimes \overline{F}_{r,' 1}$ and  since projective modules form a tensor ideal due to rigidity we have
\begin{corollary}
	The modules $F_{r, s} \boxtimes \overline{F}_{r', s'}$ are projective for all $r, s, r', s'$. 
\end{corollary}
As an example we have 
\begin{equation}\label{fus2}
\begin{split}
\overline{F}_{r', 1}  \boxtimes  F_{r, 2}  &=\overline{F}_{r', 1}  \boxtimes  (F_{r, 1} \boxtimes M_{1, 2}) = (\overline{F}_{r', 1}  \boxtimes  F_{r, 1}) \boxtimes M_{1, 2} \\
&=  \bigoplus_{\substack{\ell = 1 \\ \ell \ \text{odd}}}^p  P_{r+r'-1, \ell} \boxtimes M_{1, 2} \\
&= P_{r+r', p} \oplus P_{r+r'-2, p} \oplus   \bigoplus_{\substack{\ell = 2 \\ \ell \ \text{even}}}^p  2 P_{r+r'-1, \ell} 
\end{split}
\end{equation}
\begin{corollary}
	The modules $V_{r, s} \boxtimes \overline{V}_{r', s'}$ are projective for all $r, s, r', s'$. 
\end{corollary}
\begin{proof}
	The induction functor $\cF$ is monoidal and maps projectives to projectives.
\end{proof}
In particular the fusion rules 
\begin{equation}
\begin{split}
\overline{V}_{r, 1}  \boxtimes  V_{1, 1}  = \bigoplus_{\substack{\ell = 1 \\ \ell \ \text{odd}}}^p  R_{r, \ell} \qquad \text{and} \qquad 
\overline{V}_{r, 1}  \boxtimes  V_{1, 2}  = 2R_{r+1, p} \oplus   \bigoplus_{\substack{\ell = 2 \\ \ell \ \text{even}}}^p  2 R_{r, \ell} 
\end{split}
\end{equation}
hold via induction from \eqref{fus1} and \eqref{fus2}. We thus see that every indecomposable projective object appears in the fusion product of some $V_{r, s} \boxtimes \overline{V}_{r', s'}$. Let $\cV$ be the subcategory of $\cC_{\cW(p)}$ whose objects are direct sums of the modules of type $V_{r, s}$ and $\overline{\cV}$ the subcategory whose modules are direct sums of the $\overline{V}_{r, s}$. Then this observation can be rephrased as
\begin{corollary}\label{cor:projective}
	$\cV \boxtimes \overline{\cV}$ is a projective generator of $\cC_{\cW(p)}$.
\end{corollary}

\subsection{Free field realizations}

Let $\pi$ be the rank one Heisenberg algebra with generator $\alpha(z)$ satisfying 
\[
\alpha(z)\alpha(w) = (z-w)^{-2}.
\]
Let $\pi_\lambda$ be the Fock module of highest-weight $\lambda$ and let $L = \alpha_+  \mathbb Z$. Then the singlet and triplet algebras are subalgebras of $\pi$ and $V_L$ characterized as
\[
\cW(p) = \text{ker}( e_0^{\alpha_-\alpha}: V_L \rightarrow V_{L+\alpha_-}), \qquad \cM(p) = \text{ker}(e_0^{\alpha_-\alpha}: \pi \rightarrow \pi_{\alpha_-}).
\]
Let us denote these embeddings by $\iota$, 
\[
\iota: \cW(p) \hookrightarrow V_L, \qquad \iota: \cM(P) \hookrightarrow \pi. 
\]
Especially the triplet is an extension of the singlet, namely
\begin{equation}
\begin{split}
\cW(p) &= \bigoplus_{\lambda \in L } \text{ker}(e_0^{\alpha_-\alpha}: \pi_\lambda \rightarrow \pi_{\lambda+\alpha_-}) \\
&= \bigoplus_{n \in \mathbb Z} \text{ker}(e_0^{\alpha_-\alpha}: \pi_{\alpha_{2n+1, 1}} \rightarrow \pi_{\alpha_{2n+2, p-1}}) \\
&= \bigoplus_{n \in \mathbb Z}  M_{2n+1, 1}.
\end{split}
\end{equation}
Here we used that
\[
M_{r, s} = \text{ker}(e_0^{\alpha_-\alpha}: \pi_{\alpha_{r, s}} \rightarrow \pi_{\alpha_{r+1, p-s}}).
\]
In fact $\pi_{\alpha_{r, s}}$ satisfies the non-split exact sequence
\[
0 \rightarrow M_{r, s} \rightarrow \pi_{\alpha_{r, s}} \rightarrow M_{r+1, p-s} \rightarrow 0,
\]
that is $\pi_{\alpha_{r, s}} \cong F_{r, s}$ as $\cM(p)$-modules. 
The triplet vertex algebra is strongly generated by fields $W^\pm, W^0$ together with a Virasoro field. 
While the singlet is strongly generated by $W^0$ and the Virasoro field. The field $W^+$ is the field associated to the top level vector of $M_{3, 1}$ in $\pi_{\alpha_{3, 1}}$ and the field associated to $W^-$ is the top level vector of $M_{-1, 1}$ in $\pi_{\alpha_{-1, 1}}$.

These vertex algebras have automorphisms and we can twist these embeddings by an automorphism. 
The full automorphism group of the triplet is $PSL(2, \mathbb C)$ \cite{ALM} and especially there is an involution $\sigma$ corresponding to the non-trivial Weyl reflection. The fields $W^\pm, W^0$ carry the adjoint representation of $PSL(2, \mathbb C)$ and the superscript indicates the weight. In particular we can normalize these strong generators such that $\sigma(W^\pm) = W^\mp$ and $\sigma(W^0) = -W^0$. Note that the Virasoro field is invariant under $\sigma$. 
We see that $\sigma$ restricts to an automorphism of the singlet algebra as well. We denote by $M^\sigma$ the $\sigma$-twisted singlet module corresponding to $M$. Since $\sigma(W^\pm) = W^\mp$ and $W^+$ corresponds to the top level of $M_{3, 1}$ and $W^-$ to the top level of $M_{-1, 1}$ we have $M^\sigma_{3, 1} \cong M_{-1, 1}$.
Let $\iota^\sigma := \sigma\circ\iota$.
We denote by $\pi^\sigma_\lambda$ the Fock-module $\pi_\lambda$ viewed as an $\cM(p)$-module via the embedding $\iota^\sigma$. 
We have 
\[
0 \rightarrow M^\sigma_{r, s} \rightarrow \pi^\sigma_{\alpha_{r, s}} \rightarrow M^\sigma_{r+1, p-s} \rightarrow 0,
\]
and since $\sigma$ leaves the Virasoro subalgebra invariant one necessarily has $M^\sigma_{r, s} \cong M_{r, s}$ as modules for the Virasoro algebra. 
This fixes $M_{r, s}^\sigma \in \{ M_{r, s}, M_{2-r, s}\}$. 
We claim that the second case happens. 
\begin{proposition}
	$M_{r, s}^\sigma \cong M_{2-r, s}$ for all $r, s$. Especially $\pi^\sigma_{\alpha_{r, s}} \cong \overline{F}_{r, s}$.
\end{proposition}
\begin{proof}
	For this we note that some fusion rules of the singlet algebra have been computed using only the free field realization in section 3.2 of \cite{CMY}. The exact same argument as the one of the proof of Proposition 3.2.4 of \cite{CMY} applies to get
	\[
	M_{3, 1} \boxtimes \pi^\sigma_{\alpha_{r, s}} \cong M_{-1, 1}^\sigma \boxtimes \pi^\sigma_{\alpha_{r, s}} \cong \pi^\sigma_{\alpha_{r-2, s}}.
	\]
	Using rigidity we thus get by tensoring the non-split short exact sequence for $\pi^\sigma_{\alpha_{r, s}}$  with $M_{3, 1}$ the non-split exact sequence
	\[
	0 \rightarrow M_{3, 1} \boxtimes M^\sigma_{r, s} \rightarrow \pi^\sigma_{\alpha_{r-2, s}} \rightarrow M_{3, 1} \boxtimes M^\sigma_{r+1, p-s} \rightarrow 0,
	\]
	and hence $M_{3, 1} \boxtimes M^\sigma_{r, s} \cong M_{r-2 s}^\sigma \in \{ M_{r-2, s}, M_{4-r, s}\}$ but $M_{3, 1} \boxtimes M_{r, s} \cong M_{r+2, s}$ and 
	$M_{3, 1} \boxtimes M_{2-r, s} \cong M_{4-r, s}$ so that the only possibility is the claim $M_{r-2, s}^\sigma \cong M_{4-r, s}$.  
\end{proof}
Let $V_{\alpha_{r, s}+L}$ be the $V_L$-module corresponding to the coset $L+\alpha_{r, s}$ of $L$. It becomes an $\cW(p)$-module via the embeddings $\iota$ and $\iota^\sigma$ and we denote these by $V_{\alpha_{r, s}+L}$ and $V^\sigma_{\alpha_{r, s}+L}$ respectively. By Frobenius reciprocity we have 
\[
\text{Hom}_{\cC_{\cW(p)}}(V_{r, s} , V_{\alpha_{r, s}+L}) \cong  \text{Hom}_{\cC_{\cW(p)}}(\cF(F_{r, s}) , V_{\alpha_{r, s}+L}) 
=   \text{Hom}_{\text{Ind}(\cC^0_{\cM(p)})}(F_{r, s}, \bigoplus_{n \in \mathbb Z}  F_{r+2n, s} ) = \mathbb C
\] 
and hence $V_{r, s} \cong V^\sigma_{\alpha_{r, s}+L}$. The same argument for $\overline{V}_{r, s}$ and $V^\sigma_{\alpha_{r, s}+L}$ gives $\overline{V}_{r, s} \cong V^\sigma_{\alpha_{r, s}+L}$.
The embeddings $\iota$ and $\iota^\sigma$ provide $V_L$ with the structure of a commuative algebra in $\cC_{\cW(p)}$ and hence there are induction functors from $\cC_{\cW(p)}$ to the category of $V_L$-modules that lie in $\cC_{\cW(p)}$. This category is larger than the category $\cC_{V_L}$ of vertex algebra modules for $V_L$, but it contains $\cC_{V_L}$ as tensor catgory. Induction is monoidal and so we have oplax tensor functors from $\cC_{V_L}$ that we denote by $\cG$ and $\cG^\sigma$ that map the lattice vertex algebra module $V_{L+\alpha_{r, s}}$ to $V_{r, s}$ respectively $\overline{V}_{r, s}$, i.e. we have
\begin{proposition}\label{prop:oplax}
	There are surjective oplax tensor functors $\cG : \cC_{V_L} \rightarrow \cV$ and $\cG^\sigma : \cC_{V_L} \rightarrow \overline{\cV}$.
\end{proposition}

\section{Categorical Setup}

\subsection{Module Coalgebras}
We include here the fundamental definitions and results on module coalgebras required in the following subsection. The content of this subsection can be found in greater detail in \cite[Subsection 4.2]{BCPO}. Throughout, we denote the coassociator, coproduct, and unit of a quasi-bialgebra $H$ by $\Phi_H$, $\Delta_H$, and $\epsilon_H$.

\begin{definition}
Let $H$ be a quasi-bialgebra and $C$ a left $H$-module. We call $C$ a left $H$-module coalgebra if it is equipped with morphisms
\[ \delta:C \to C \otimes C, \qquad \varepsilon:C \to \mathbb{C} \]
subject to the following conditions:
\begin{align}
\Phi_H \cdot (\delta \otimes \mathrm{Id}_C) \circ \delta(c)&=(\mathrm{Id}_C \otimes \delta) \circ \delta(c) \\
\label{Hcou}\varepsilon(c_1)c_2&=\varepsilon(c_2)c_1=c \\
\label{Hcop}\delta(h \cdot c)&=\Delta_H(h) \cdot \delta(c),\\
\varepsilon(h \cdot c)&=\epsilon_H(h)\varepsilon(c),
\end{align}
for all $c \in C$ and $h \in H$, where we are using Sweedler's notation $\delta(c)=c_1 \otimes c_2$. Right $H$-module coalgebras are defined similarly. 
\end{definition}

We also consider coalgebra objects in categories of bimodules over two distinct Hopf algebras.
\begin{definition}
Let $H,G$ be Hopf algebras and $C$ a $H$-$G$ bimodule. We call $C$ a $H$-$G$ bimodule coalgebra if is equipped with morphisms
\[ \delta:C \to C \otimes C, \qquad \varepsilon:C \to \mathbb{C} \]
satisfying
\begin{align}
\label{HGcoass}\Phi_H \cdot (( \delta \otimes \mathrm{Id}_C) \circ \delta(c) ) \cdot \Phi_G^{-1}&=(\mathrm{Id}_C \otimes \delta) \circ \delta(c)\\
\delta(h \cdot c)=\Delta_H(h) \cdot \delta(c), \; & \; \delta(c \cdot g)=\delta(c) \cdot \Delta_G(g)\\
\varepsilon(c_1)c_2&=\varepsilon(c_2)c_1=c\\
\varepsilon(h \cdot c)=\epsilon_H(h)\varepsilon(c), \; & \; \varepsilon(c \cdot g)=\varepsilon(c)\epsilon_G(g)
\end{align}
\end{definition}
We have the following useful result (\cite[Proposition 4.17]{BCPO}):
\begin{proposition}\label{Gaugeprop}
Let $H$ be a quasi-bialgebra, $F \in H \otimes H$ a gauge transformation on $H$, and $(C,\delta,\varepsilon)$ a left $H$-module coalgebra. If we define $\delta_F(c)=F\delta(c)$ for all $c \in C$, then $(C,\delta_F,\varepsilon)$ is an $H^F$-module coalgebra, where $H^F$ is the $F$-twist of $H$.
\end{proposition}

\subsection{Categorical Tools} \label{catsec}
We proceed first by deriving some general results on quasi Hopf algebras which will be necessary for the following sections.

\begin{proposition}\label{prop1}
	Let $\mathcal{C}$ be a finite tensor category with integral Frobenius-Perron dimensions and $U$ an algebra such that we have an abelian equivalence $\cC \cong \mathrm{Rep}(U)$ identifying Frobenius-Perron dimensions in $\cC$ with vector space dimensions in $\mathrm{Rep}(U)$. Then, $U$ can be given the structure of a quasi bialgebra such that $\mathcal{C}$ and $\mathrm{Rep}(U)$ are tensor equivalent.
\end{proposition}

\begin{proof}
	Since $\cC$ is a finite tensor category with integral Frobenius-Perron dimensions, there exists a quasi-fiber functor $F:\cC \to \mathrm{Vec}$ and a quasi bialgebra $\tilde{U}:=\mathrm{End}(F)$ such that $\cC \cong \mathrm{Rep}(\tilde{U})$ as tensor categories. It remains only to show that $\tilde{U} \cong U$ as algebras. However, we have an abelian isomorphism $\mathrm{Rep}(U) \cong \mathrm{Rep}(\tilde{U})$ identifying vector space dimensions with Frobenius-Perron dimensions and it follows from the proof of the reconstruction theorem for finite-dimensional quasi bialgebras \cite[Theorem 5.3.12]{EGNO} that the algebra structure of $\tilde{U}$ is uniquely determined by the abelian structure of $\mathrm{Rep}(\tilde{U})$ and Frobenius-Perron dimensions (i.e. the associated quasi-fiber functor).
\end{proof}

\begin{proposition}\label{prop2} 
	Let $\cC, \cD$ be a finite tensor categories such that there exists a quasi-fiber functor $F_{\cD}:\cD \to \mathrm{Vec}$ and an essentially surjective quasi-tensor functor $G:\cC \to \cD$. Then we have an inclusion $C \subset U$ of the associated quasi bialgebra $U=\mathrm{End}(F_{\cD} \circ G)$, $C=\mathrm{End}(F_{\cD})$ preserving the coproduct, counit, and algebra structure. 
\end{proposition}

\begin{proof}
	We have \cite[Subsection 1.10]{EGNO}
	\begin{align*}
	\mathrm{Coend}(F_{\cD})&= \bigoplus\limits_{Y \in \cD} F_{\cD}(Y)^* \otimes F_{\cD}(Y)/E_{\cD}\\
	\mathrm{Coend}(F_{\cD} \circ G)&= \bigoplus\limits_{X \in \cC} F_{\cD}(G(X))^* \otimes F_{\cD}(G(X))/E_{\cC}
	\end{align*}
	where $E_{\cD}$ is spanned by the elements of the form $y_* \otimes F_{\cD}(f)x-F_{\cD}(f)^*y_*\otimes x$ for $x \in F_{\cD}(X)$, $y_* \in F_{\cD}(Y)^*$, $f \in \mathrm{Hom}_{\cD}(X,Y)$ and $E_{\cC}$ is spanned by elements of the form $y_* \otimes F_{\cD}(G(g))^*y_* \otimes x$ for $x \in F_{\cD}(G(X))$, $y_* \in F_{\cD}(G(Y))^*$, $g \in \mathrm{Hom}_{\cC}(X,Y)$. It is clear that $E_{\cC} \subset E_{\cD}$ and $G$ is essentially surjective, so we have a canonical surjection 
	\[ \mathrm{Coend}(F_{\cD} \circ G) \twoheadrightarrow \mathrm{Coend}(F_{\cD})\]
	which yields an inclusion
	\[ \mathrm{End}(F_{\cD}) \hookrightarrow \mathrm{End}(F_{\cD} \circ G)\]
	For a finite tensor $\cB$ and quasi-fiber functor $F:\cB \to \mathrm{Vec}$, the coproduct and counit on an element $a$ of $\mathrm{End}(F)$ is defined uniquely by the action of $a$ on $F(X \otimes Y)$ and $F(1)$ viewed as modules for $\mathrm{End}(F)$ for any $X,Y \in \cB$ \cite[Theorem 5.2.1]{EGNO} so the inclusion preserves the coproduct and counit. 
\end{proof}

\begin{proposition} \label{bimod}~
	\begin{itemize}
		\item[a)] Let $U$ be a quasi bialgebra. If there exists a module coalgebra structure $(U_{reg},\delta)$ on the regular representation $U_{reg}$ of $U$ such that $\delta(1)$ is invertible in $U \otimes U$, then there is a twist equivalent quasi bialgebra $U^J$ such that $\Phi^{J}=1 \otimes 1 \otimes 1$ and $\delta^J$ is the coproduct on $U^J$. 
		\item[b)] Let $U$ be a quasi bialgebra and $C \subset U$ a subalgebra. Let $(U,\delta)$ be a $U$-$C$-bimodule coalgebra such that its structure as a $U$-$C$ bimodule is that of the regular representation of $U$. If $\delta(1)$ is invertible, then there exists a twist $J$ of $U$ such that $\Phi_{U^J}=\Phi_C$ as operators on $U$.
	\end{itemize}
\end{proposition}
\begin{proof}
	a) Since $\delta(1)$ is invertible in $U \otimes U$, it follows from Equation \eqref{Hcou} that $\delta(1)$ is a gauge transformation on $U$ so we can define the twisted quasi bialgebra $U^J$ where $J:=\delta(1)^{-1}$. It then follows from Proposition \ref{Gaugeprop} that $(U_{reg},\delta^J)$ is a module-coalgebra for $U^J$ where $\delta^J(x)=J\delta(x)$ for all $x \in U^J$, so $\delta^J(1)=1 \otimes 1$. Since $\delta^J$ is coassociative, we have $\Phi^J(\delta^J \otimes \mathrm{Id}) \circ \delta^J = (\mathrm{Id} \otimes \delta^J) \circ \delta^J$ so $\Phi^J(1 \otimes 1) \otimes 1=1 \otimes (1 \otimes 1)$ and because $U_{reg}$ is a faithful representation, we have $\Phi^J=1 \otimes 1 \otimes 1$. It follows from Equation \eqref{Hcop} that $\delta^J$ is the coproduct on $U^J$.\\
	
	b) As in a), we have a gauge transformation $J:=\delta(1)^{-1}$ such that $\delta^J(1)=1 \otimes 1$ and $(U_{reg},\delta^J)$ is a $U^J$-$C$ bimodule coalgebra. Coassociativity gives
	\[ \Phi_{U^J} \cdot (\delta^J \otimes \mathrm{Id}) \circ \delta(x) \cdot \Phi_C^{-1}=(\mathrm{Id} \otimes \delta^J) \circ \delta^J(x)\]
	Taking $x=1$ gives 
	\[ \Phi_{U^J} \cdot (1 \otimes 1) \otimes 1 \cdot \Phi_C^{-1}=(1 \otimes 1) \otimes 1 \]
	so $\Phi_{U^J}=\Phi_C$ since we are in the regular representation.
\end{proof}
\begin{proposition}\label{oplax}
	Suppose we have oplax monoidal functors $\mathbb{V}_k:\mathrm{Rep}(C) \to \mathrm{Rep}(U)$, $k=1,...,n$ such that $\bigotimes\limits_{k=1}^n \mathbb{V}_k(C_{reg})$ is a projective generator of $\mathrm{Rep}(U)$, such that in addition the coproduct $\delta$ induced by the coproduct of the regular $C$-$C$-bimodule $C_{reg}$ sends the image of a generator $1$ to a $U$-$U$-bimodule generator. Then there exists a twist $J$ of $U$ such that $\Phi_{U^J}=\Phi_{C}$.
\end{proposition}

\begin{proof}
	We can view $C$ as a $C$-$C$-bimodule coalgebra with the regular representation. The category of $C$-bimodules is equivalent to the relative Deligne tensor category $\mathrm{Rep}(C) \boxtimes_{\mathrm{Rep}(C)} \mathrm{Rep}(C)$. Since the image of a coalgebra object under an oplax functor is a coalgebra object, the image of $C$ under $\mathbb{V}_K \boxtimes \mathrm{Id}$ is a $U$-$C$ bimodule coalgebra for each $k=1,...,n$ with $\delta(1)$ invertible. Therefore, it follows from Proposition \ref{bimod} that the gauge transformation $J:=\delta(1)$ yields a twist $U^J$ such that $\Phi_{U^J}$ and $\Phi_C$ agree as operators on $\mathbb{V}_k(C)$. However, 
	\[ \bigotimes_{k=1}^n \mathbb{V}_k(C)\]
	is a projective generator for $\mathrm{Rep}(U)$, so $\Phi_{U^J}$ and $\Phi_C$ agree as operators on a projective generator and therefore agree.
\end{proof}

\begin{theorem}\label{Catthm}
	Let $\cC, \cD$ be a finite tensor categories such that there exists a quasi-fiber functor $F_{\cD}:\cD \to \mathrm{Vec}$ and an essentially surjective quasi-tensor functor $G:\cC \to \cD$ as in Proposition \ref{prop2}. 	Suppose we have oplax monoidal functors $\mathbb{V}_k:\mathrm{Rep}(C) \to \mathrm{Rep}(U)$, $k=1,...,n$ such that $\bigotimes_k \mathbb{V}_k(C_{reg})$ is a projective generator of $\mathrm{Rep}(U)$, such that in addition the coproduct $\delta$ induced by the coproduct of the regular $C$-$C$-bimodule $C_{reg}$ sends the image of a generator $1$ to a $U$-$U$-bimodule generator, as in Proposition \ref{oplax}.
	Then there exist a twist $U^J$ of $U$, such that $\Phi_{U^J}=\Phi_C$ and thus $C=\mathrm{End}(F_{\cD})$ is a sub quasi bisubalgebra of $U^J$. Moreover there is a unique choice of such a twist such that $\delta$
\end{theorem} 

\begin{proof}
	We have seen in Proposition \ref{prop2} that there is an inclusion $C \subset U$ of the associated quasi bialgebra $U=\mathrm{End}(F_{\cD} \circ G)$, $C=\mathrm{End}(F_{\cD})$ preserving the coproduct, counit, and algebra structure. Since the conditions of Propositiion \ref{oplax} are satisfied, we also have $\Phi_{U^J}=\Phi_C$ for some twist $J$ of $U$. 
\end{proof}

\section{Characterization}
Our goal is to classify all quasitriangular quasi bialgebras whose representation categories coincide with that of the triplet vertex operator algebra $\mathcal{W}(2)$. We proceed in steps, first classifying the possible coproducts and antipodes. That is, we first classify those quasi bialgebras whose representation categories are tensor equivalent to that of the triplet. We then classify the quasitriangular quasi bialgebras whose representation categories are braided tensor equivalent to that of the triplet. \\

\begin{definition}
	As an algebra, the small quantum group $\tilde{u}_i(\mathfrak{sl}_2)=u_i(\mathfrak{sl}_2)$ at fourth root of unity $i=\sqrt{-1}$ is the $\mathbb{C}$-algebra with generators $E,F,K^{\pm 1}$, and relations
	\begin{align}
	\label{Uq1} KE&=-EK, & KF&=-FK, & [E,F]&=\frac{1}{2}(K-K^{-1})\\
	\label{Uq2} \quad & & F^2&=E^2=0, & KK^{-1}&=K^{-1}K=1
	\end{align}
\end{definition}
\noindent $\tilde{u}_i(\mathfrak{sl}_2)$ has four Verma modules $M_k$ with $k\in \{0,1,2,3\}$. Each $M_k$ is two dimensional spanned by $\{v_k,Fv_k\}$ where
\[ Kv_k=i^kv_k, \quad KFv_k=-i^kFv_k, \quad Ev_k=0, \quad EFv_k=\delta_{2|k}i^{k-1} v_k \]
where $\delta_{2|k}$ is $1$ if $2|k$ and $0$ otherwise. It is clear then that $M_k$ is irreducible if $k$ is odd and reducible if $k$ is even. Let $X_k$, $k\in \{0,...,3\}$ denote the irreducible modules of $\tilde{u}_i(\mathfrak{sl}_2)$ where $X_k$ is the irreducible quotient of $M_k$. We therefore have $X_k=M_k$ if $k$ is odd and 
\[ K \cdot X_k = i^k \cdot X_k\]

\noindent The following semisimple quasi-triangular quasi-Hopf algebra $C$ appears below as the Cartan subalgebra of $u_i(\mathfrak{sl}_2)$:

\begin{definition}
Let $C=\C[\Z_4]=\mathrm{Span}_{\mathbb{C}}\{K\}$ with $K^4=1$ be a group algebra. This algebra is isomorphic to the dual group algebra $\C^{\Z_4}=\bigoplus\limits_{k=0}^3 \mathbb{C}_k $ via 
\begin{align*}
e_0&:=\frac{1}{4}(1+K+K^2+K^3) & e_1&:=\frac{1}{4}(1-iK-K^2+iK^3)\\
e_2&:=\frac{1}{4}(1-K+K^2-K^3) & e_3&:=\frac{1}{4}(1+iK-K^2-iK^3)
\end{align*}
which are orthogonal idempotents $ \label{eij} e_i \cdot e_j= \delta_{i,j} e_i$.\\

\noindent
The algebra $C$ has accordingly four simple representations $\C_i$ with the action
\begin{equation}\label{Kaction} K \cdot \mathbb{C}_k = i^k \mathbb{C}_k.
\end{equation}
The algebra $C$ can be given the structure of a quasi-triangular quasi-Hopf algebra with coproduct, counit, antipode
\begin{align}
\label{Ccop}\Delta_C(e_k)&=\sum\limits_{i+j=k} e_i \otimes e_j\\
\label{Ccou}\epsilon_C(e_k)&=\delta_{k,0}\\
\label{Cant}S_C(e_k)&=e_{-k}\\
 \intertext{and with coassociator and $R$-matrix given by a quadratic form $Q(k)=\beta^{(k^2)}$ on $\Z_4$ for any fixed $\beta$ with $\beta^4=-1$. We take the following representative associator and $R$-matrix in accordance with \cite{FGR1}}\\
\label{Ccoass}\Phi_C^{\pm1}&= \sum\limits_{a,b,c=1}^3 \phi_{abc}^\pm (e_a \otimes e_b \otimes e_c) \\
\phi_{abc}^\pm&:=\begin{cases}
\mp \beta^2,\quad &2\nmid a,b\text{ and } c=1\\
-1,\quad &2\nmid a,b\text{ and } c=2\\
\pm \beta^2,\quad &2\nmid a,b\text{ and } c=3\\
1, \quad &\mathrm{otherwise}
\end{cases}\\
R&=\sum\limits_{a,b=1}^3 R^{a,b}(e_a\otimes e_b)\\
R^{a,b}&:=\begin{pmatrix}
1 & 1 & 1 & 1 \\
1 & \beta & 1 & \beta \\
1 & -1 & -1 & 1 \\
1 & -\beta & -1 & \beta 
\end{pmatrix},
	R^{a,b}R^{b,a}=\begin{pmatrix}
1 & 1 & 1 & 1 \\
1 & \beta^2 & -1 & -\beta^2 \\
1 & -1 & 1 & -1 \\
1 & -\beta^2  & -1 & \beta^2 
\end{pmatrix}
\end{align}
\end{definition}

\begin{definition}
The oplax tensor functors $\V,\overline{\V}$ from $\Rep(C)$ to $\Rep(U)$ are defined by induction 
\begin{align*}
\V(X)&:=X\otimes_{U^{\leq 0}} U \\
\overline{\V}(X)&:=X \otimes_{U^{\geq 0}} U\\
\end{align*}
We have $\mathbb{V}(\mathbb{C}_k)=\mathrm{Span}\{e_k,Fe_k\}$ where $Ee_k=0$, and it follows from Equations \eqref{Kaction} and \eqref{Uq1} that
\begin{align*}
EFe_0&=0, & EFe_1&=e_1, & EFe_2&=0,  & EFe_3&=-e_3.
\end{align*}
Similarly, $\overline{\mathbb{V}}(\mathbb{C}_k^{\pm})=\mathrm{Span}\{e_k^{\pm}, Ee_k^{\pm}\}$ where $Fe_k^{\pm}=0$ and 
\begin{align*} FEe_0^+&=0, &  FEe_1&=-e_1, & FEe_2&=0,  & FEe_3&=e_3.
\end{align*}
 We therefore see that $\mathbb{V}(e_k) \cong M_k$, $\overline{\mathbb{V}}(e_k) \cong \overline{M}_k$ as modules over $u_i(\mathfrak{sl}_2)$ where $\overline{M}_k$ is the opposite Verma module. Then the images of the regular representation of $C$ under $\mathbb{V},\overline{\mathbb{V}}$ are as $u_i(\mathfrak{sl}_2)$-modules
\[ \mathbb{V}(C_{reg}) \cong \bigoplus\limits_{k=0}^3 M_k, \qquad \overline{\mathbb{V}}(C_{reg}) \cong \bigoplus\limits_{k=0}^3 \overline{M}_k\] 
 We now consider the regular $C$-$C$-bimodule $C_{\reg}=\bigoplus_{k=0}^3 \C_{k}$ and the induced module $\V(C_\reg)=\bigoplus \V(\C_k)$ with 
basis $\{e_k,Fe_k\}$, $k\in \{0,1,2,3\}$. The left action of $C=\C^{\Z_4}$ is by 
$C\subset U$ and the right action of $C$ conveniently coincides in this 
notation with the right-multiplication by $C\subset U$.
\end{definition}

\subsection{Classification of braided tensor structure on $\Rep(U)$}

The image of a coalgebra under an oplax tensor functor is a coalgebra. Similarly, the image of a $C$-$C$-bimodule coalgebra under $\V,\overline{\V}$ is a $U$-$C$-bimodule coalgebra. Note that the regular representation $C_{reg}$ is a  $C$-$C$-bimodule coalgebra but not a $C$-module coalgebra due to the nontrivial associator. We now classify all possible coalgebra structures on $\V(C_\reg),\overline{\V}(C_\reg)$. In this explicit case it follows for $1=\sum_ae_a$ from the grading and the coassociativity that $\delta(1)$ is again a generator. Hence we may invoke Theorem \ref{Catthm} and thus assume that after going to a twist equivalent $U$ we have $\delta(1)=1\otimes 1$ and $\Phi_U=\Phi_C$.

\begin{proposition}\label{prop_Vreg}
     The following are all $\Z_4$-$\Z_4$-graded coalgebra structures on 
$\V(C_\reg)$ such that $\delta(1)=1\otimes 1$.

     \[ \delta(F \cdot1) := (F\otimes 1)\sum_{a,b} c_L^{ab}(e_a\otimes e_b)
     + (1\otimes F)\sum_{a,b} c_R^{ab}(e_a\otimes e_b)\]
with $c_L^{ab},c_R^{a,b} \in \mathbb{C}$. If $c^{a,b}_L, c^{a,b}_R\neq 0$, then
\begin{align}
\label{cL}c_L^{ab}&=\frac{c_{a+b}}{c_a}\\
\label{cR}c_R^{ab}&=\epsilon^a \frac{c_{a+b}}{c_b}(-1)^{a(a-1)/2}
\end{align}
     for any choice of scalars $c_0=1,c_1,c_2,c_3$ and any global choice 
$\epsilon^4=1$. By the same argument, $\delta(E \cdot1)$ is determined on 
$\overline{V}(C_\reg)$ by the same equation, swapping $F$ and $E$ with coefficients
$\bar{c}_0=1,\bar{c}_1,\bar{c}_2,\bar{c}_3$ and $\bar{\epsilon}^4=1$.
\end{proposition}

\begin{proof}
We have assumed $\delta(1)=1\otimes 1$ where $1=\sum\limits_{k=0}^3 e_k$ which gives
\[ \sum\limits_{k=0}^3 \delta(e_k)=\delta(1)=1 \otimes 1 =\sum\limits_{a,b=0}^3 e_a \otimes e_b. \]
Since $\mathbb{V}(C_{\reg})$ is $\mathbb{Z}_4$-graded as a coalgebra with $\mathbb{V}(C_{\reg})_{\bar{k}}=\mathrm{Span}\{e_k,Fe_k\}$, we have
$$ \delta(e_k)=\sum_{a+b=k} e_{a}\otimes e_{b}$$
Since the coproduct must respect the $\mathbb{Z}_4$-grading, an arbitrary expression for $\delta(F \cdot e_k)$ is
\[ \delta(F \cdot e_k)=\sum\limits_{\substack{ u,v \in \{0,1\} \\ a+b=k}} F^ue_a \otimes F^v e_b, \]
and $K \otimes K \delta(F \cdot e_k)= -i^k \delta( F \cdot e_k)$, so we have $(-1)^{u+v}i^k=-i^k$ and therefore $u+v$ is odd. Hence, an arbitrary expression for $\delta(F \cdot 1)=\sum\limits_{k=0}^3\delta(F \cdot e_k)$ is

\begin{align*}
\delta(F \cdot 1)&=: (F\otimes 1)\sum_{a,b=0}^3 c_L^{ab}(e_a\otimes e_b)
+ (1\otimes F)\sum_{a,b=0}^3 c_R^{ab}(e_a\otimes e_b)\\
\intertext{and thus by the right $C$-module structure}
\delta(F \cdot e_k)&= (F \otimes 1) \sum\limits_{a+b=k} c_L^{ab} (e_a \otimes e_b)+(1 \otimes F) \sum\limits_{a+b=k}c_R^{ab}(e_a \otimes e_b)
\end{align*}
By applying counitality to each $\delta(F \cdot e_k)$ we have $c_L^{a,0}=c_R^{0,b}=1$. We now compute constraints on the coefficients using coassociativity


\begin{align*}
&(\delta \otimes \mathrm{Id}) \circ \delta(F \cdot 1)=\sum\limits_{a,b=0}^3 c_L^{ab}\delta(F \cdot e_a) \otimes e_b+\sum\limits_{a,b=0}^3 c_R^{ab} \delta(e_a) \otimes F \cdot e_b \\
&= \sum\limits_{a,b=0}^3 \sum\limits_{u+v=a} c_L^{ab} c_L^{uv}  F  e_u \otimes e_v \otimes e_b\\
&+\sum\limits_{a,b=0}^3 \sum\limits_{u+v=a} c_L^{ab}c_R^{uv}e_u \otimes Fe_v \otimes e_b+\sum\limits_{a,b=0}^3\sum\limits_{u+v=a}c_R^{ab} e_u \otimes e_v \otimes Fe_b\\
\qquad \\
&(\mathrm{Id} \otimes \delta) \circ \delta(F \cdot 1)=\sum\limits_{a,b=0}^3 c_L^{ab}Fe_a \otimes \delta(e_b)+ \sum\limits_{a,b=0}^3 c_R^{ab}e_a \otimes \delta(Fe_b)\\
&=\sum\limits_{a,b=0}^3 \sum\limits_{u+v=b} c_L^{ab}Fe_a \otimes e_u \otimes e_v \\
&+\sum\limits_{a,b=0}^3 \sum\limits_{u+v=b} c_R^{ab}c_L^{uv} e_a \otimes Fe_u \otimes e_v+\sum\limits_{a,b=0}^3 \sum\limits_{u+v=b} c_R^{ab}c_R^{uv}e_a \otimes e_u \otimes Fe_v\\
\end{align*}


Here, we have to use the coassociator of $U$-$C$-bimodules, which we have 
proven in Theorem \ref{Catthm} coincides with the coassociator of $C$ from Equation \eqref{Ccoass}. By Equation \eqref{HGcoass}, the coassociator acts on $U$-$C$-bimodules by $X \mapsto \Phi \cdot X \cdot \Phi^{-1}$, which acts trivially on $(e_a\otimes 
e_b\otimes e_c)$, $(Fe_a\otimes e_b\otimes 
e_c)$, and $(e_a\otimes Fe_b \otimes e_c)$, while it acts as $(-1)^{ab}$ on $(e_a\otimes e_b\otimes Fe_c)$. Now coassociativity 
reads for all $a,b,c$:
\begin{align*}
c_L^{ab}c_L^{a+b,c} &= c_L^{a,b+c}\\
  c_R^{ab}c_L^{a+b,c} &=c_L^{bc}c_R^{a,b+c}\\
  (-1)^{ab}c_R^{a+b,c} &=c_R^{bc}c_R^{a,b+c}
\end{align*}

\noindent The first and third equations with $a=0$ and $c=0$ respectively give 
\begin{align*}
c_L^{bc} &=\frac{c_L^{0,b+c}}{c_L^{0,b}}, \\
c_R^{ab} &=\frac{c_R^{a+b,0}}{c_R^{b,0}}(-1)^{ab}.
\end{align*}
This conversely solves the equations for any values 
$c_a:=c_L^{0,a},\;d_a:=c_R^{a,0}$, where $c_0=d_0=1$. If we plug this into the second equation we get
$$(-1)^{ab}\frac{d_{a+b+c}d_b}{d_{a+b}d_{b+c}}
=\frac{c_{a+b+c}c_b}{c_{a+b}c_{b+c}}.$$
In particular $b=0$ is a $2$-cocycle condition
$$(-1)^{ac}\frac{d_{a+c}}{d_{a}d_{c}}
=\frac{c_{a+c}}{c_{a}c_{c}}$$
which gives a necessary expression for $d_n$ in terms of $c_n$ and 
$\epsilon:=d_1/c_1$
\begin{equation} \label{bn} d_n=\epsilon^n c_n(-1)^{n(n-1)/2}
\end{equation}
and the necessary condition $\epsilon^4=1$. Rewriting the coefficients in terms of $c_n$ then gives $c_L^{ab}=\frac{c_{a+b}}{c_a}$ and
\begin{align*}
c_R^{ab}&=\frac{(-1)^{ab}d_{a+b}}{d_b}\\
&=(-1)^{ab}\frac{\epsilon^{a+b}c_{a+b}(-1)^{(a+b)(a+b-1)/2}}{\epsilon^b c_b (-1)^{b(b-1)/2}}\\
&=\epsilon^a \frac{c_{a+b}}{c_b}(-1)^{a(a-1)/2}
\end{align*}
\end{proof}

We now use our knowledge of the coalgebras $\V(C_{reg}),\overline{\V}(C_{reg})$ to determine the coproduct on $U$. This proceeds in several steps. Some of the following arguments are more general, but for clarity we spell them out on the explicit algebra in question:
\begin{enumerate}[a)]
\item We write down an arbitrary expression for $\Delta(E)$ and $\Delta(F)$, 
keeping in mind that $K$ acts adjointly on $E,F$ by $-1$. Therefore, since $\Delta(KX)=-\Delta(XK)$ for $X\in \{E,F\}$ and $\Delta(K)=K \otimes K$, the summands in $\Delta(E)$ and $\Delta(F)$ 
have to contain $E,F$ an odd number of times
\begin{align*}
\Delta(F)
&=\sum\limits_{a,b =0}^3\big( (F \otimes 1)  c_{L}^{ ab}+(1 
\otimes F)  c_{R}^{ ab}\\
&+(E \otimes 1)  r_{L}^{ab}+(1 \otimes E)  r_{R}^{ ab}\\
&+(E \otimes EF)  t_{L}^{ab}+(EF \otimes E)  t_{R}^{ab}\\
&+(F \otimes EF)  s_{L}^{ab}+(EF \otimes F)  s_{R}^{ab}\big) (e_a 
\otimes e_b )\\
\quad\\   
\Delta(E)
&=\sum\limits_{a,b =0}^3\big( (E \otimes 1) \bar{c}_{L}^{ ab}+(1 
\otimes E)  \bar{c}_{R}^{ab}\\
&+(F \otimes 1)  \bar{r}_{L}^{ab}+(1 \otimes F)  \bar{r}_{R}^{ ab}\\
&+(F \otimes EF)  \bar{t}_{L}^{ab}+(EF \otimes F) \bar{t}_{R}^{ab}\\
&+(E \otimes EF)  \bar{s}_{L}^{ab}+(EF \otimes E) \bar{s}_{R}^{ab}\big) 
(e_a \otimes e_b )
\end{align*}
\item It follows from $\Delta(X) \cdot (1 \otimes 1) = \Delta(X) \cdot \delta(1)=\delta(X \cdot 1)$ that  the coefficients $c_L,c_R,\bar{c}_L,\bar{c}_R$ coincide with the corresponding coefficients of $\delta(F \cdot 1)$ appearing in Proposition \ref{prop_Vreg}. The existence of a right quasi-antipode requires that $c_L^{a,-a}$ is invertible, and it follows from the coassociativity relations in the proof of Proposition \ref{prop_Vreg} that $c^{a,-a}c^{0,a}=c^{a,0}=1$, so the additional nonzero condition in Proposition \ref{prop_Vreg} is fulfilled, that is, $c_L^{ab}$ and $c_R^{ab}$ are non-zero. On the other hand, since we have a zero action $E \cdot 1= F \cdot \bar{1} =0$, we have $r^{ab}_L=r^{ab}_R=\bar{r}^{ab}_L=\bar{r}^{ab}_R=0$.\\

\item The coefficient of $E e_a \otimes E e_0 \otimes F e_b$ in $(\Delta \otimes \mathrm{Id}) \circ \Delta(F)$ is zero, but the coefficient of the same term in $(\mathrm{Id} \otimes \Delta) \circ \Delta(F)$ is $t_L^{ab}\bar{c}_L^{0,b}c_R^{0,b+2}$. We therefore see that $t_L^{a,b}=0$ since $c_L^{a,b}$ and $c_R^{a,b}$ are non-zero. Similarly, inspecting the coefficients of $E e_a \otimes Fe_0 \otimes Ee_b$, $Ee_a \otimes Fe_0 \otimes Fe_b$, and $Fe_a \otimes Ee_0 \otimes Fe_b$ respectively show that $t_R^{a,b}$, $s_R^{a,b}$, and $s_L^{a,b}$ all vanish. Applying the same argument to $\Delta(E)$ shows that the corresponding coefficients for $\Delta(E)$ also vanish. Therefore, we have
\[ t_L^{a,b}=t_R^{a,b}=s_L^{a,b}=s_R^{a,b}=\bar{t}_L^{a,b}=\bar{t}_R^{a,b}=\bar{s}_L^{a,b}=\bar{s}_R^{a,b}=0 \\ \]

\item We now use the nilpotency of $F$ to further restrict the coefficients:
\begin{align*}
0=\Delta(F)^2
&= \big((F\otimes 1)\sum_{a,b=0}^3 c_L^{ab}(e_a\otimes e_b)
+ (1\otimes F)\sum_{a,b} c_R^{ab}(e_a\otimes e_b)\big)^2\\
&=\sum_{a,b=0}^3 (c_L^{a,b+2}c_R^{a,b}+c_R^{a+2,b}c_L^{a,b})(F\otimes 
F)(e_a\otimes e_b)\\
%
\end{align*}
Now, since the coefficients are non-zero and have the form described in Proposition \ref{prop_Vreg}, we have:
\begin{align*}
 \quad 0
&=\frac{c_{a+b+2}}{c_{a}}\cdot  \epsilon^a 
(-1)^{a(a-1)/2}\frac{c_{a+b}}{c_{b}}
+\epsilon^{a+2} (-1)^{(a+2)(a+1)/2}\frac{c_{a+b+2}}{c_{b}}\cdot 
\frac{c_{a+b}}{c_a}
\end{align*}
so $\epsilon^2=1$ in Proposition \ref{prop_Vreg}. Similarly 
$\Delta(E)^2=0$ implies $\bar{\epsilon}^2=1$. 
\item We now check the commutation relations for $E$ and $F$, $[\Delta(E),\Delta(F)]=K(e_1+e_3)$.
\begin{align*}
&\Delta(E)\Delta(F)-\Delta(F)\Delta(E)\\
&=\left(\sum_{a',b'}\big((E \otimes 1)  \bar{c}_{L}^{ a'b'}+(1 \otimes 
E)  \bar{c}_{R}^{a'b'}\big)(e^{a'} \otimes e^{b'} )\right)
\cdot\left(\sum_{a,b}\big((F \otimes 1)  c_{L}^{ ab}+(1 \otimes F)  
c_{R}^{ ab}\big)(e^{a} \otimes e^{b} )\right)\\
&-\left(\sum_{a',b'}\big((F \otimes 1)  c_{L}^{ a'b'}+(1 \otimes F)  
c_{R}^{ a'b'}\big)(e^{a'} \otimes e^{b'} )\right)
\cdot \left(\sum_{a,b}\big((E \otimes 1)  \bar{c}_{L}^{ ab}+(1 \otimes 
E)  \bar{c}_{R}^{ab}\big)(e^{a} \otimes e^{b} )\right)\\
&=\sum_{a,b}\big((EF\otimes 1)\bar{c}_L^{a+2,b}c_L^{a,b}
+(F\otimes E)\bar{c}_R^{a+2,b}c_L^{a,b}
+(E\otimes F)\bar{c}_L^{a,b+2}c_R^{a,b}
+(1\otimes EF)\bar{c}_R^{a,b+2}c_R^{a,b}\\
&-(FE\otimes 1)c_L^{a+2,b}\bar{c}_L^{a,b}
-(E\otimes F)c_R^{a+2,b}\bar{c}_L^{a,b}
-(F\otimes E)c_L^{a,b+2}\bar{c}_R^{a,b}
-(1\otimes FE)c_R^{a,b+2}\bar{c}_R^{a,b}\big)(e^{a} \otimes e^{b} )\\
\intertext{and on the other hand}
&\Delta \left(K(e_1+e_3)\right)
=\sum\limits_{a+b=1} Ke_a \otimes Ke_b + \sum\limits_{a+b=3} Ke_a \otimes Ke_b\\
&= \sum\limits_{a,b=0}^3 i^{a+b}(\delta_{2\mid a} \delta_{2 \nmid b}+\delta_{2 \nmid a} \delta_{2 \mid b})e_a \otimes e_b
\end{align*}
The vanishing of the mixed term $F\otimes E$ reads  
\begin{align*}
\bar{c}_R^{a+2,b}c_L^{a,b}&=c_L^{a,b+2}\bar{c}_R^{a,b}\\
\bar{\epsilon}^{a+2} (-1)^{(a+2)(a+1)/2}\frac{\bar{c}_{a+b+2}}{\bar{c}_{b}}
\cdot \frac{c_{a+b}}{c_{a}} 
&=\frac{c_{a+b+2}}{c_{a}}
\cdot \bar{\epsilon}^a (-1)^{a(a-1)/2}\frac{\bar{c}_{a+b}}{\bar{c}_{b}}\\
%
%
\Rightarrow\quad -\bar{c}_{t+2}/\bar{c}_{t}&=c_{t+2}/c_t
\end{align*}
so in particular $-\bar{c}_2=c_2$). The vanishing of the mixed term $E\otimes F$ gives the same quasiperiodicity relation. \\
\begin{align*}
\bar{c}_L^{a,b+2}c_R^{a,b}&=c_R^{a+2,b}\bar{c}_L^{a,b} \\
\frac{\bar{c}_{a+b+2}}{\bar{c}_{a}} 
\cdot \epsilon^a (-1)^{a(a-1)/2}\frac{c_{a+b}}{c_{b}}
&=\epsilon^{a+2} (-1)^{(a+2)(a+1)/2}\frac{c_{a+b+2}}{c_{b}}
\cdot \frac{\bar{c}_{a+b}}{\bar{c}_{a}} \\
\end{align*}

\noindent
The term $(EF-FE)\otimes 1 $ requires first that the two coefficients are equal
\begin{align*}
\bar{c}_L^{a+2,b}c_L^{a,b}
&=c_L^{a+2,b}\bar{c}_L^{a,b}\\
\frac{\bar{c}_{a+b+2}}{\bar{c}_{a+2}}\frac{c_{a+b}}{c_{a}}  
&=\frac{c_{a+b+2}}{c_{a+2}}\frac{\bar{c}_{a+b}}{\bar{c}_{a}}  
\end{align*}
which is always true by the quasiperiodicity relation. \\

\noindent
The term $1\otimes (EF-FE)$ requires first that the two coefficients are equal
\begin{align*}
\bar{c}_R^{a,b+2}c_R^{a,b}
&=c_R^{a,b+2}\bar{c}_R^{a,b}\\
\bar{\epsilon}^a (-1)^{a(a-1)/2}\frac{\bar{c}_{a+b+2}}{\bar{c}_{b+2}}
\cdot \epsilon^a (-1)^{a(a-1)/2}\frac{c_{a+b}}{c_{b}}
&=\epsilon^a (-1)^{a(a-1)/2}\frac{c_{a+b+2}}{c_{b+2}}
\cdot \bar{\epsilon}^a (-1)^{a(a-1)/2}\frac{\bar{c}_{a+b}}{\bar{c}_{b}}
\end{align*} 
which is always true by the quasiperiodicity relation. \\

Now we apply $[E,F]=K(e_1+e_3)$ which acts on $e_a$ by $i^a\delta_{2\nmid a}$ and compute the joint result of $(EF-FE)\otimes 1$ and $1\otimes (EF-FE)$, then the final equation to be fulfilled is:
\begin{align*}
i^a\delta_{2\nmid a}
\cdot \frac{\bar{c}_{a+b+2}}{\bar{c}_{a+2}}\frac{c_{a+b}}{c_{a}}  
+i^b\delta_{2\nmid b}
\cdot \bar{\epsilon}^a\frac{\bar{c}_{a+b+2}}{\bar{c}_{b+2}}
\cdot \epsilon^a\frac{c_{a+b}}{c_{b}}
&=i^a\delta_{2\mid a}i^b\delta_{2\nmid b}+i^a\delta_{2\nmid a}i^b\delta_{2\mid b}
\end{align*}
For $2\mid a,b$ both sides are zero. For $2\nmid a,b$ 
\begin{align*}
i^a
\cdot \frac{\bar{c}_{a+b+2}}{\bar{c}_{a+2}}\frac{c_{a+b}}{c_{a}}  
+i^b (\epsilon\bar{\epsilon})^a
\cdot\frac{\bar{c}_{a+b+2}}{\bar{c}_{b+2}}
\frac{c_{a+b}}{c_{b}}
&=0 \\
i^a
\cdot \frac{1}{\bar{c}_{a+2}c_{a}}  
+i^b(\epsilon\bar{\epsilon})^a
\cdot\frac{1}{\bar{c}_{b+2}c_{b}}
&=0
\end{align*} 
which for $a=b$ (odd) means $\epsilon\bar{\epsilon}$ and which for $b=a+2$ means $\bar{c}_{1}c_3=-\bar{c}_{3}c_1$ so it holds again by the quasiperiodicity relation.\\

\noindent
For $2\nmid a,2\mid b$ resp.  $2\mid a,2\nmid b$ 
\begin{align*}
\frac{\bar{c}_{a+b+2}}{\bar{c}_{a+2}}\frac{c_{a+b}}{c_{a}}  
&=i^b\\
\frac{\bar{c}_{a+b+2}}{\bar{c}_{b+2}}
\frac{c_{a+b}}{c_{b}}
&=i^a 
\end{align*}
which are equivalent after switching $a,b$, so we continue with the first. The first relation always holds for $b=0$, and the case $b=2$ always holds by the quasiperiodicity relation.
\end{enumerate}

\noindent
The previous steps thus yield:

\begin{proposition}\label{prop_EF}
     The coproduct is necessarily
     \begin{align*}
     \Delta(F)
     &=\sum\limits_{a,b \in \{\pm \}}\big( (F \otimes 1) c_{L}^{ ab}+(1 
\otimes F)  c_{R}^{ ab}\big)(e_a\otimes e_b)\\
     \Delta(E)
     &=\sum\limits_{a,b \in \{\pm \}}\big( (E \otimes 1) \bar{c}_{L}^{ 
ab}+(1 \otimes E)  \bar{c}_{R}^{ab}\big)(e_a\otimes e_b)
     \end{align*}
     where $c_{L}^{ ab}, c_{R}^{ ab}$ and  $\bar{c}_{L}^{ 
ab},\bar{c}_{R}^{ ab}$ are defined in terms of 
$c_a,\epsilon,\bar{c}_a,\bar{\epsilon}$ as in Equations \eqref{cL} and \eqref{cR} under the following additional assumptions:
     $$\epsilon^2=\bar{\epsilon}^2=1,\quad 
\epsilon\bar{\epsilon}=-1,\quad c_{t+2}/c_t=-\bar{c}_{t+2}/\bar{c}_{t}$$
     Conversely, every such coproduct defines a quasi-bialgebra 
structure on $U$ with respect to the associator $\Phi=\Phi_C$.
\end{proposition}

\noindent We now apply an algebra automorphism to remove the parameters 
$c_a,\bar{c}_a$:\\

\begin{proposition}~\label{prop_automorphism}
     \begin{enumerate}[a)]
         \item For any $x_0=1,x_1,x_2,x_3\in\C^\times$ the element 
$F':=F(\sum_a x_ae_a)$ also fulfills  the coproduct formula in 
Proposition \ref{prop_Vreg}, with modified parameters $c'_a=c_ax_a$.
         \item The assignment
         \begin{align*}
         K&\mapsto K\\
         E&\mapsto E':=E(\sum_a \bar{x}_a e_a)\\
         F&\mapsto F':=F(\sum_a x_a e_a)
         \end{align*}
         extends to an algebra automorphism of $U$ iff
         $$\bar{x}_1=x_3^{-1},\quad
         \bar{x}_2=x_2,\quad
         \bar{x}_3=x_1^{-1}$$
         \item Applying the algebra automorphism for $x_i=c_i$ yields 
the coproducts
         \begin{align*}
         \Delta(F)&=F\otimes 1+\omega_{-\epsilon}\otimes F\\
         \Delta(E)&=(E\otimes 1)(\sum_{a,b}(\bar{c}')_L^{ab}e_a\otimes 
e_b)+(\omega_{\epsilon}\otimes E)
         (\sum_{a,b}(\bar{c}')_L^{ba}e_a\otimes e_b)\\
\omega_{\epsilon}&:=\sum_a(-\epsilon)^a(-1)^{a(a-1)/2}e_a=((e_0+e_2)+i\epsilon 
(e_1+e_3))K \\
         (\bar{c}')_L^{ab}&= \begin{pmatrix}
         1 & d & -1 & -d \\
         1 & -d^{-1} & -1 & d^{-1} \\
         1 & d & -1 & -d \\
         1 & -d^{-1} & -1 & d^{-1} \\
         \end{pmatrix}
         \end{align*} 
     \end{enumerate}
\end{proposition}
\begin{example} 
	In particular for $d=\pm i, \epsilon=1$ the quasi-bialgebra above is the quasi-Hopf algebra in \cite{FGR2} Section 3.1
	\begin{align*}
	\Delta(F)&=F\otimes 1+\omega_{-}\otimes F\\
	\Delta(E)&=E\otimes K^{\pm 1}+\omega_{+}K^{\pm 1}\otimes E\\
	\omega_{\pm}&:=\sum_a(-\epsilon)^a(-1)^{a(a-1)/2}=((e_0+e_2)\pm i\epsilon 
	(e_1+e_3))K
	\end{align*} 
	with the conventions $F=\mathsf{f^-},E=K\mathsf{f^+}$ and central idempotents $e_0+e_2=\mathsf{e_0},e_1+e_3=\mathsf{e_2}$.
	
	For $\epsilon=-1$ we get an isomorphic quasi-Hopf algebra with $\mathsf{f^-},\mathsf{f^+}$ switched, and for $d=-i$ we get  an isomorphic quasi-Hopf algebra  with switched $K,K^{-1}$ and a different initial coassociator $\Phi$ for $\beta'=i\beta$.
	
\end{example}
\begin{proof}[Proof of Proposition \ref{prop_automorphism}]
\begin{enumerate}[a)]
     \item By multiplicativity of the coproduct
     \begin{align*}
     \Delta(F(\sum_a x_ae_a))
     &=(F\otimes 1)\sum_{a,b} c_L^{ab}(e_a\otimes e_b)(\sum_{a',a''} 
x_{a'+a''} e_{a'}\otimes e_{a''}) \\
     &+ (1\otimes F)\sum_{a,b} c_R^{ab}(e_a\otimes e_b)(\sum_{a',a''} 
x_{a'+a''} e_{a'}\otimes e_{a''})\\
     &=(F(\sum_a x_ae_a)\otimes 1)\sum_{a,b} 
c_L^{ab}\frac{x_{a+b}}{x_a}(e_a\otimes e_b)\\
     &+ (1\otimes F(\sum_b x_ae_b))\sum_{a,b} 
c_R^{ab}\frac{x_{a+b}}{x_b}(e_a\otimes e_b)
     \end{align*}
     and $c_L^{ab},c_R^{ab}$ are proportional to 
$\frac{c_{a+b}}{c_a},\frac{c_{a+b}}{c_b}$ respectively.\\

     \item Nilpotency and commutation relations with $K$ surely hold for 
the new $E',F'$, it remains to check $[E',F']=Ke_1$:
     \begin{align*}
     [E',F']
     &=E(\sum_a \bar{x}_a e_a)F(\sum_a {x}_a e_a)
     -F(\sum_a {x}_a e_a)E(\sum_a \bar{x}_a e_a)\\
     &=EF(\sum_a \bar{x}_{a+2}x_a e_a)
     -FE(\sum_a {x}_{a+2}\bar{x}_a e_a)
     \end{align*}
     The two factors are equal iff 
$\bar{x}_{a+2}/\bar{x}_{a}={x}_{a+2}/{x}_{a}$, so iff $\bar{x}_2=x_2$ 
and $\bar{x}_3=\bar{x}_1(x_3/x_1)$, and they are equal to $1$ for 
$2\nmid a$ (where $[E,F]\neq0$) iff $\bar{x}_{a+2}x_a=1$, so iff 
$\bar{x}_{3}x_1=1$ and $\bar{x}_{1}x_3=1$, which then holds automatically.\\

     \item If we set $x_i=c_i^{-1}$ the coproduct of $F'$ is as asserted
     $$\Delta(F')=\sum_{a,b}\big((F\otimes 1)+(1\otimes 
F)\epsilon^a(-1)^{a(a-1)/2}\big)(e_a\otimes e_b)$$
     On the other hand by the choice of $\bar{x}_i$ in b):
     $$\bar{c}_1'=\bar{c}_1c_3,\;
     \bar{c}_2'=\bar{c}_2/c_2,\;
     \bar{c}_3'=\bar{c}_3c_1$$
     If we introduce $d=\bar{c}_1c_3$ then by the quasiperiodicity 
relation $-d=\bar{c}_3c_1$, on the other hand $\bar{c}_2/c_2=-1$. Going 
through all $a,b$ gives the matrix
     $$(\bar{c}')_L^{ab}=\frac{\bar{c}'_{a+b}}{\bar{c}'_a}
     = \begin{pmatrix}
     1 & d & -1 & -d \\
     1 & -d^{-1} & -1 & d^{-1} \\
     1 & d & -1 & -d \\
     1 & -d^{-1} & -1 & d^{-1} \\
     \end{pmatrix}$$
\end{enumerate}
\end{proof}

\begin{proposition}\label{braiding}
	The previous bialgebra for $\epsilon=1$ has an $R$-matrix iff $d=\pm i$, and it is of the form 
	\begin{align*}
	R
	&=\sum\limits_{a,b =0}^3\bigg( R^{a,b}(1 \otimes 1)+R^{a,b}_{F,E}(F \otimes E)\bigg)(e^a \otimes e^b)
	R^{a,b}
	&=\begin{pmatrix}
	1 & 1 & 1 & 1 \\
	1 & \beta & 1 & \beta \\
	1 & -1 & -1 & 1 \\
	1 & -\beta & -1 & \beta
	\end{pmatrix},\\
	R_{F,E}^{a,b}
	&=2\begin{pmatrix}
	\mp 1 & -i & \mp 1 & -i \\
	\mp 1 & -i\beta & \mp 1 & -i\beta \\
	\mp 1 & i & \pm  & -i \\
	\mp 1 & i\beta & \pm 1 & -i\beta \\
	\end{pmatrix}
	\end{align*}
\end{proposition} 
\begin{proof}
	An arbitrary element in $u_q \otimes u_q$ is a $u^0_q$-combination of $E^iF^j \otimes E^kF^l$ with $i,j,k,l \in \{0,1\}$. We restrict the form in the following steps (the order of which chosen to minimize terms in computations) 
	\begin{enumerate}[a)]
		\item Compatibility of the $R$-matrix with the coproduct ($R\Delta_{U(e_0)}(K)=\Delta^{op}(K)R$) means, since $XK=-KX$ for $X \in \{E,F\}$,  that for each term in $R$, we must have $2|(i+j+k+l)$. Therefore,
		\begin{align*}
		R&=\sum\limits_{a,b =0}^3\bigg( R^{a,b}(1 \otimes 1)+R^{a,b}_{E,F}(E \otimes F)+R_{F,E}^{a,b}(F \otimes E)+R^{a,b}_{E,E}(E \otimes E)+R^{a,b}_{F,F}(F \otimes F)\\
		&+R_{EF,1}^{a,b}(EF \otimes 1)+R_{1,EF}^{a,b}(1 \otimes EF)+R_{EF,EF}^{a,b}(EF \otimes EF) \bigg)(e^a \otimes e^b)
		\end{align*} 
		\item Counitality $(\epsilon\otimes 1)(R)=(1\otimes\epsilon)(R)=1$ implies  
		$$R^{0b}=R^{a0}=1,\quad R_{EF,1}^{a0}=R_{1,EF}^{0b}=0$$
		\item The functors $\V,\overline{\V}$ are braided, and the tensor product $\V(\C_a)\otimes \V(\C_b)$ decomposes into the images, 
		 so $R^{a,b}=\sigma(a,b)$ is the braiding of $C$: 
		\item Writing out the left hexagon identity $(\Delta\otimes 1)(R)=R_{13}R_{23}$ implies $$R^{a,b}_{EF,1}=R^{a,b}_{1,EF}=0$$ 
		because the terms $E\otimes F\otimes 1$ on the left hand side cannot appear on the right hand side, and similarly for the right hexagon. We later write out the full hexagon identity.
		\item Compatibility of the $R$ matrix with the coproduct of $F$ reads, using also $FE=EF-e_1K$:	
		\begin{align*}
		R\Delta(F)&=\sum\limits_{a,b}\bigg( R^{a+2,b}c_L^{ab}(F \otimes 1)+R_{E,F}^{a+2,b}c_L^{ab}(EF \otimes F)\\
		&+R^{a,b+2}c_R^{ab}(1 \otimes F)+R_{F,E}^{a,b+2}c_R^{ab}(F \otimes EF) \bigg) (e^a \otimes e^b)\\
		\Delta(F)^{\tau}R
		&=\sum\limits_{a,b} \bigg( c_L^{ba}R^{a,b}(1 \otimes F)+c_L^{b+2,a+2}R_{F,E}^{a,b}(F \otimes (EF-e_1K))\\
		&+c_R^{ba}R^{a,b}(F \otimes 1)+c_R^{b+2,a+2}R_{E,F}^{a,b}((EF-e_1K) \otimes F) \bigg) (e^a \otimes e^b)
		\end{align*} 
		Comparing coefficients then gives the following relations
		\begin{align*}
		R^{a+2,b}c_L^{ab}&=c_R^{ba}R^{a,b}+c_L^{b+2,a+2}R_{F,E}^{a,b}(-\delta_{2\nmid b}i^b)\\
		R^{a,b+2}c_R^{ab} &=c_L^{ba}R^{a,b}+c_R^{b+2,a+2}R_{E,F}^{a,b}(-\delta_{2\nmid a}i^a)\\
		R_{E,F}^{a+2,b}c_L^{ab} 
		&= c_R^{b+2,a+2}R_{E,F}^{a,b}
		\\
		R_{F,E}^{a,b+2}c_R^{ab}
		&=c_L^{b+2,a+2}R_{F,E}^{a,b}
		\end{align*}
		
		We now assume the standard form of the coproduct achieved in Proposition \ref{prop_automorphism} $$c_L^{ab}=1, c_R^{ab}=\epsilon^a(-1)^{a(a-1)/2}, \epsilon=1$$ 
		then the previous relations read explicitly
		\begin{align*}
		R_{F,E}^{a,b}(-\delta_{2\nmid b}i^b)
		&=R^{a+2,b}-R^{a,b}(-1)^{b(b-1)/2}=2\begin{pmatrix}
		0 & -1 & 0 & 1 \\
		0 & -\beta & 0 & \beta \\
		0 & 1 & 0 & 1 \\
		0 & \beta & 0 & \beta \\
		\end{pmatrix}\\
		R_{F,E}^{a,b}
		&=2\begin{pmatrix}
		* & -i & * & -i \\
		* & -i\beta & * & -i\beta \\
		* &  i& * & -i \\
		* &  i\beta & * & -i\beta \\
		\end{pmatrix}\\
		R_{E,F}^{a,b}(\delta_{2\nmid a}i^a) &=R^{a,b}-R^{a,b+2}(-1)^{a(a-1)/2}
		=\begin{pmatrix}
		0 & 0 & 0 & 0 \\
		0 & 0 & 0 & 0 \\
		0 & 0 & 0 & 0 \\
		0 & 0 & 0 & 0 \\
		\end{pmatrix}\\
		R_{E,F}^{a,b}
		&=\begin{pmatrix}
		* & * & * & * \\
		0 & 0 & 0 & 0 \\
		* & * & * & * \\
		0 & 0 & 0 & 0 \\
		\end{pmatrix}\\
		\intertext{The other two equations are fulfilled on the visible part of $R_{E,F}^{a,b},R_{F,E}^{a,b}$ and fix the unknown third row/column relative to the unknown first row/column}
		R_{E,F}^{a+2,b}
		&= (-1)^{b(b-1)/2}R_{E,F}^{a,b}
		\\
		R_{F,E}^{a,b+2}
		&=(-1)^{a(a-1)/2}R_{F,E}^{a,b}
		\end{align*}		
		
		\item Compatibility of the $R$ matrix with the coproduct of $E$ reads
		\begin{align*}
		R\Delta(E)&=\sum\limits_{a,b}\bigg( R^{a+2,b}\bar{c}_L^{ab}(E \otimes 1)+R_{F,E}^{a+2,b}\bar{c}_L^{ab}((EF-e_1K) \otimes E)\\
		&+R^{a,b+2}\bar{c}_R^{ab}(1 \otimes E)+R_{E,F}^{a,b+2}\bar{c}_R^{ab}(E \otimes (EF-e_1K)) \bigg) (e^a \otimes e^b)\\
		\Delta(E)^{\tau}R
		&=\sum\limits_{a,b} \bigg( \bar{c}_L^{ba}R^{a,b}(1 \otimes E)+\bar{c}_L^{b+2,a+2}R_{E,F}^{a,b}(E \otimes EF)\\
		&+\bar{c}_R^{ba}R^{a,b}(E \otimes 1)+\bar{c}_R^{b+2,a+2}R_{F,E}^{a,b}(EF \otimes E) \bigg) (e^a \otimes e^b)
		\end{align*} 
		Comparing coefficients then gives the following relations
		\begin{align*}
		R^{a+2,b}\bar{c}_L^{ab}+R_{E,F}^{a,b+2}\bar{c}_R^{ab}(-\delta_{2\nmid b}i^b)&=\bar{c}_R^{ba}R^{a,b}\\
		R^{a,b+2}\bar{c}_R^{ab}+R_{F,E}^{a+2,b}\bar{c}_L^{ab}(-\delta_{2\nmid a}i^a) &=\bar{c}_L^{ba}R^{a,b}\\
		R_{F,E}^{a+2,b}\bar{c}_L^{ab} 
		&= \bar{c}_R^{b+2,a+2}R_{F,E}^{a,b}
		\\
		R_{E,F}^{a,b+2}\bar{c}_R^{ab}
		&=\bar{c}_L^{b+2,a+2}R_{E,F}^{a,b}
		\end{align*}
		
		We now assume the standard form of the coproduct achieved in Proposition \ref{prop_automorphism}
		$$\bar{c}_L^{ab}= \begin{pmatrix}
		1 & d & -1 & -d \\
		1 & -d^{-1} & -1 & d^{-1} \\
		1 & d & -1 & -d \\
		1 & -d^{-1} & -1 & d^{-1} \\
		\end{pmatrix},\quad  \bar{c}_R^{ab}=\epsilon^a(-1)^{a(a-1)/2}\bar{c}_L^{ba}, \quad \epsilon=-1$$ 
		then the previous relations read explicitly
		\begin{align*}
		R_{E,F}^{a,b+2}\epsilon^a(-1)^{a(a-1)/2}\bar{c}_L^{ab}(\delta_{2\nmid b}i^b)
		&=(R^{a+2,b}-R^{a,b}(-1)^b(-1)^{b(b-1)/2})\bar{c}_L^{ab}
		=\begin{pmatrix}
		0 & 0 & 0 & 0 \\
		0 & 0 & 0 & 0 \\
		0 & 0 & 0 & 0 \\
		0 & 0 & 0 & 0 \\
		\end{pmatrix}\\
		R_{E,F}^{a,b}
		&=\begin{pmatrix}
		* & 0 & * & 0 \\
		* & 0 & * & 0 \\
		* & 0 & * & 0 \\
		* & 0 & * & 0 \\
		\end{pmatrix}\\
		R_{F,E}^{a+2,b}\bar{c}_L^{ab}(-\delta_{2\nmid a}i^a) =(R^{a,b}-&R^{a,b+2}(-1)^a(-1)^{a(a-1)/2})\bar{c}_L^{ba}
		=2\begin{pmatrix}
		0 & 0 & 0 & 0 \\
		d & -d^{-1}\beta & d & -d^{-1}\beta \\
		0 & 0 & 0 & 0 \\
		-d & -d^{-1}\beta & d & d^{-1}\beta \\
		\end{pmatrix}\\
		R_{F,E}^{a,b}
		&=
		2\begin{pmatrix}
		* & * & * & * \\
		di & -i\beta & di & -i\beta \\
		* & * & * & * \\
		di & i\beta & -di & -i\beta \\
		\end{pmatrix}\\
		\intertext{The other two equations are fulfilled on the visible part of $R_{E,F}^{a,b},R_{F,E}^{a,b}$ and fix the unknown third row/column relative to the unknown first row/column (we use the 2-periodicity behaviour $\pm1$ of $\bar{c}_L^{ab}$)}
		R_{F,E}^{a+2,b}
		&= -(-1)^{b+2}(-1)^{(b+2)(b+1)/2}R_{F,E}^{a,b}
		\\
		R_{E,F}^{a,b+2}&=-(-1)^a(-1)^{a(a-1)/2}R_{E,F}^{a,b}
		\end{align*}		
		\item Combining the partial matrices and the periodicity relations in the two previous bullets 
		we get
		\begin{align*}
		R_{E,F}^{a,b}
		&=2\begin{pmatrix}
		X & 0 & -X & 0 \\
		0 & 0 & 0 & 0 \\
		X & 0 & X & 0 \\
		0 & 0 & 0 & 0 \\
		\end{pmatrix}\\
		R_{F,E}^{a,b}
		&=2\begin{pmatrix}
		Y & -i & Y & -i \\
		di & -i\beta & di & -i\beta \\
		Y & i & -Y & -i \\
		di & i\beta & -di & -i\beta \\
		\end{pmatrix}
		\end{align*}
		
		\item	We now turn to the full hexagon identity
		$$R_1^{(1)}\otimes R_1^{(2)} \otimes R_2
		=(\Phi_{2}\otimes\Phi_{3}\otimes \Phi_{1})^{-1} (R_1\otimes 1\otimes R_2)
		(\Phi_{1}\otimes\Phi_{3}\otimes \Phi_{2}) (1\otimes R_1\otimes R_2)
		(\Phi_{1}\otimes\Phi_{2}\otimes \Phi_{3})^{-1}
		$$
		We denote $\phi_{abc}^\pm=1,\mp \beta^2,-1,\pm \beta^2$ for $a,b$ odd and $c=0,1,2,3$ and else $\phi_{abc}=1$, and we denote by $\odot$ the component-wise multiplication of matrices.\\
		
		For the term $(1\otimes 1 \otimes 1)$ the hexagon equation holds, because $R^{ab}$ is the R-matrix of $c$. 
		
		We now consider the term $(F\otimes 1\otimes E)(e_a\otimes e_b\otimes e_c)$, and 
		\begin{align*}
		c_L^{ab}R_{F,E}^{a+b,c} &=\phi_{c+2,a+2,b}^- R_{F,E}^{a,c} \phi^+_{acb} R^{b,c} \phi^-_{abc} \\
		\intertext{$b=0$ holds trivially. We spell this equation out for $b=1$ and use $\phi_{c+2,a+2,b}^-\phi_{acb}^+=1$}
		\begin{pmatrix}
		di & -i\beta & di & -i\beta \\
		Y & i & -Y & -i \\
		di & i\beta & -di & -i\beta \\
		Y & -i & Y & -i \\
		\end{pmatrix}
		&=
		\begin{pmatrix}
		Y & -i & Y & -i \\
		di & -i\beta & di & -i\beta \\
		Y & i & -Y & -i \\
		di & i\beta & -di & -i\beta \\
		\end{pmatrix}
		\odot
		\begin{pmatrix}
		1 & \beta & 1 & \beta \\
		1 & \beta & 1 & \beta \\
		1 & \beta & 1 & \beta \\
		1 & \beta & 1 & \beta \
		\end{pmatrix}
		\odot
		\begin{pmatrix}
		1  & 1 & 1 & 1 \\
		1 & \beta^2 & -1 & -\beta^2 \\
		1 & 1 & 1 & 1 \\
		1 & \beta^2 & -1 & -\beta^2 \\
		\end{pmatrix}\\
		\intertext{This equation is fulfilled for $di=Y$.}
		\intertext{The same calcuation for $E\otimes 1\otimes F$ shows clearly $X=0$ and thus $R_{E,F}^{ab}=0$.\newline
			The calculation for $E\otimes F\otimes EF$ relates $R_{EF,EF}^{a+b,c}$ to the product of $R_{E,F}$ and $R_{F,E}$, so we have $R_{EF,EF}^{a,b}=0$. Hence  $(F\otimes 1\otimes E)$ and $(1\otimes F\otimes E)$ are the only terms appearing. }
		\end{align*}

		\item We now turn to the other hexagon identity
		$$R_1\otimes R_2^{(1)}\otimes R_2^{(2)}
		=(\Phi_{3}\otimes\Phi_{1}\otimes \Phi_{2}) (R_1\otimes 1\otimes R_2)
		(\Phi_{2}\otimes\Phi_{1}\otimes \Phi_{3})^{-1} (R_1\otimes R_2\otimes 1)	
		(\Phi_{1}\otimes\Phi_{2}\otimes \Phi_{3})
		$$
		We consider the term $(F\otimes E\otimes 1)$ for $c=1$ using 
		$\phi^-_{b+2,a+2,c}\phi^+_{abc}=1$
		$$
		\bar{c}_L^{bc}R_{F,E}^{a,b+c} =\phi_{b+2,c,a+2}^+ R^{a+2,c} \phi^-_{b+2,a+2,c} R_{F,E}^{a,b} \phi^+_{abc}$$
		\begin{align*}
		&\begin{pmatrix}
		d &  d & d & d \\
		-d^{-1} & -d^{-1} & -d^{-1} & -d^{-1} \\
		d &  d & d & d \\
		-d^{-1} &  -d^{-1} & -d^{-1} & -d^{-1} \\
		\end{pmatrix}
		\odot 
		\begin{pmatrix}
		-i & Y & -i &  Y \\
		-i\beta & di & -i\beta & di\\
		i & -Y & -i & Y \\
		i\beta & -di & -i\beta & di\\
		\end{pmatrix}\\
		&=
		\begin{pmatrix}
		1 & -1 & 1 & -1 \\
		1 & \beta^2 & 1 & \beta^2 \\
		1 & 1 & 1 & 1 \\
		1 & -\beta^2 & 1 & -\beta^2 \\
		\end{pmatrix}
		\odot
		\begin{pmatrix}
		-1 & -1 & -1 & -1 \\
		-\beta & -\beta & -\beta & -\beta \\
		1 & 1 & 1 & 1 \\
		\beta & \beta & \beta & \beta \\
		\end{pmatrix}
		\odot
		\begin{pmatrix}
		Y & -i & Y & -i \\
		di & -i\beta & di & -i\beta \\
		Y & i & -Y & -i \\
		di & i\beta & -di & -i\beta \\
		\end{pmatrix}
		\intertext{These equations hold iff $d^2=-1$ (all equations with odd $a+b$).\qedhere}
		\end{align*} 
	\end{enumerate}
\end{proof}

\noindent
We formulate the results of this section:

\begin{theorem}
	The abelian category $\Rep(U)$ admits up to equivalence a unique structure of a braided tensor category such that $\V,\overline{\V}:\Rep(C) \to \Rep(U)$ are braided oplax tensor functors with the braiding on $C$ given by the quadratic form $Q(k)=\beta^{(k^2)}$, for fixed $\beta^4=-1$.
\end{theorem}

\noindent It follows from Theorem \ref{thm:oplax} that the triplet algebra $\mathcal{W}(2)$ admits two such oplax tensor functors $\V,\overline{\V}:\Rep V_L \to \Rep \mathcal{W}(2)$ where the $Q$ is determined by the quadratic form of the lattice $L=\sqrt{2} A_1$, so $\beta=e^{2\pi i(1/8)}$. It then follows that the abelian category $\mathrm{Rep}(U)$ of the Kazhdan-Lusztig dual $U$ of $\mathcal{W}(2)$ must admit the functors described in Theorem \ref{classif}. We therefore obtain as a corollary a proof of the Kazhdan-Lusztig conjecture for $\mathcal{W}(2)$:
\begin{corollary}
There is a braided tensor equivalence of  $\Rep(\mathcal{W}(2))$  to the representation category of the previously constructed quasi-Hopf algebra for this value of $\beta$. The quasi Hopf algebra coincides with the one constructed in \cite{FGR2}.
\end{corollary}


\newcommand\arxiv[2]      {\href{http://arXiv.org/abs/#1}{#2}}

\end{document}